\documentclass[11pt]{article}

\usepackage{amsthm}
\usepackage{amsmath}
\usepackage{amssymb}
\usepackage{amsfonts}
\usepackage{epsfig}
\usepackage{graphicx}
\usepackage{color}
\usepackage{soul}
\usepackage{float}
\usepackage[caption = false]{subfig}
\usepackage{algorithm,algorithmic}
\usepackage{multirow}
\usepackage[table]{xcolor}

\newtheorem{theorem}{Theorem}
\newtheorem{proposition}{Proposition}
\newtheorem{lemma}{Lemma}

\newtheorem{remark}{Remark}

\allowdisplaybreaks

\begin{document}

\title{\bf Optimizing chemoradiotherapy to target multi-site metastatic disease and tumor growth}
\author{Hamidreza Badri,${}^{1}$ Ehsan Salari,${}^{2}$ Yoichi Watanabe,${}^{3}$ Kevin Leder${}^{1}$\\
{\small ${}^{1}$ Department of Industrial and Systems Engineering, University of Minnesota, Minneapolis, MN}\\{\small ${}^{2}$ Department of Industrial and Manufacturing Engineering, Wichita State University, Wichita, KS}\\{\small ${}^{3}$ Department of Radiation Oncology,  University of Minnesota, Minneapolis, MN}\\}

\date{\today} 

\normalsize{}
\maketitle

\begin{abstract}
The majority of cancer-related fatalities are due to metastatic disease \cite{gupta2006cancer}. In chemoradiotherapy (CRT), chemotherapeutic agents are administered along with radiation to {control} the primary tumor and systemic disease such as metastasis. This work introduces a mathematical model of CRT treatment scheduling to obtain optimal drug and radiation protocols with the objective of minimizing metastatic cancer-cell populations at multiple potential sites while maintaining a desired level of {control} on the primary tumor. A dynamic programming (DP) framework is used to determine the optimal radiotherapy fractionation regimen and the drug administration schedule. {\color{black}We design efficient DP data structures and use structural properties of the optimal solution to reduce the complexity of the resulting DP algorithm.} Also, we derive closed-form expressions for optimal chemotherapy schedules in special cases. {\color{black}The results suggest that if there is only an additive and spatial cooperation between the chemotherapeutic drug and radiation with no interaction between them, then radiation and drug administration schedules can be decoupled. In that case, regardless of the chemo- and radio-sensitivity parameters, the optimal radiotherapy schedule follows a hypo-fractionated scheme.} However, the structure of the optimal chemotherapy schedule depends on model parameters such as chemotherapy-induced cell-kill at primary and metastatic sites, as well as the ability of primary tumor cells to initiate successful metastasis at different body sites. {\color{black} In contrast, an interactive cooperation between the drug and radiation leads to optimal split-course concurrent CRT regimens. Additionally, under dynamic radio-sensitivity parameters due to the re-oxygenation effect during therapy, we observe that it is optimal to immediately start the chemotherapy and administer few large radiation fractions at the beginning of the therapy, while scheduling smaller fractions in later sessions.} We quantify the trade-off between the new and traditional objectives of minimizing the metastatic population size and maximizing the primary tumor control probability, respectively, for a cervical cancer case. The trade-off information indicates the potential for significant reduction in the metastatic population with minimal loss in the primary tumor control.

 
\end{abstract}
 
\section{Introduction}

Metastasis occurs when cancer cells spread from the primary tumor site to anatomically distant locations. The process of cancer metastasis consists of the following steps: tumor cells detach from the primary tumor site, invade a blood or lymphatic vessel, transit in the bloodstream, and extravasate from the blood vessels to distant sites where they proliferate to form new colonies. It is thought that cancer cells evolve a special set of traits that improve their ability to carry out this process \cite{chaffer2011perspective,farhidzadeh2014prediction,farhidzadeh2016signal}. Single or multi-site metastasis is a common occurrence in cancer patients. For instance, bone metastasis occurs in up to 70\% of patients with advanced prostrate or breast cancer \cite{roodman2004mechanisms}, and 20\% of patients with carcinoma of the uterine cervix develop single ($32\%$) or multi-site ($68\%$) metastasis at distant locations such as lymph nodes, lung, bone, or abdomen \cite{carlson1967distant}. The development of metastases significantly lessens the chances of successful therapy and represents a major cause of mortality in cancer patients. For instance, only 20\% of patients with breast cancer \cite{roodman2004mechanisms} and less than 1\% of patients with carcinoma of the uterine cervix \cite{carlson1967distant} are still alive five years after the discovery of metastasis.

Many researchers have used mathematical modeling and various optimization techniques to find clinically relevant optimal radiation delivery schedules. The mathematical building block for the majority of radiotherapy response modeling is the \emph{linear-quadratic} (LQ) model \cite{hall2006radiobiology}. This model states that following exposure to $d$ Gy of radiation, the surviving fraction of tumor cells is given by $e^{-\alpha d-\beta d^2}$, where $\alpha$ and $\beta$ are cell-specific parameters. In a conventional radiotherapy fractionation problem, the goal is to maximize the biological effect of radiation on the tumor while inflicting the least amount of {healthy-tissue} damage, which are evaluated using \emph{tumor control probability} (TCP) and \emph{normal-tissue complication probability} (NTCP), respectively \cite{wigg2001applied}. Some important questions concern the best total treatment size, the best way to divide the total dose into fractional doses, and the optimal inter-fraction interval times. An important result emerging from recent work is that for low values of {tumor $\left[\alpha/\beta\right]$}, a hypo-fractionated schedule, where large fraction sizes are delivered over a small number of treatment days, is optimal. However for large values of $\left[\alpha/\beta\right]$, a hyper-fractionated schedule, where small fraction sizes are delivered over a large number of treatment days, is optimal \cite{mizuta2012mathematical,badri2015optimal,badriglioma,bortfeld2013optimization}. 

For many inoperable tumors, radiotherapy alone is not enough to successfully control the primary tumor. For these tumors, chemoradiotherapy (CRT) is the standard of care, in which one or several chemotherapeutic agents are administered along with radiation. There are at least two basic mechanisms by which the combination of the two modalities may result in a therapeutic gain. First, chemotherapy may be effective against (occult) systemic disease, such as metastasis. This mechanism is called \emph{spatial cooperation}. Second, the two modalities act independently {(so-called \emph{additivity}) or dependently (so-called \emph{radio-sensitization}) to enhance tumor cell-kill at the primary site \cite{steel1979exploitable}}. Recently, Salari et al. \cite{salari2013mathematical} studied changes in optimal radiation fractionation regimens that result from adding chemotherapeutic agents to radiotherapy when maximizing TCP. More specifically, they incorporated additivity and radio-sensitization mechanisms into the radiation fractionation decision to identify corresponding changes in optimal radiation fractionation schemes. Their results suggest that chemotherapeutic agents with only an additive effect do not impact optimal radiation fractionation schemes; however, radio-sensitizers may change the optimal fractionation schemes and in particular may give rise to optimal nonstationary schemes.

In our very recent work \cite{badri2015minimizing}, we introduced a novel {biologically driven} objective function that takes into account metastatic risk instead of maximizing TCP. We observed that when designing optimal radiotherapy treatments with the goal of minimizing metastatic production, hypo-fractionation is preferable for many parameter choices, and {standard fractionation} is only preferable when the $\left[\alpha/\beta\right]$ value of the tumor is sufficiently large and the long-term metastasis risk (several years after therapy) is considered. Our prior results \cite{badri2015minimizing} demonstrated a proof of concept that radiation scheduling decisions have the potential to significantly impact metastatic risk. However, that work excluded several important features in the modeling of advanced carcinomas. 

First, we only considered single-site metastatic disease. However, it is well known that some tumors, e.g., lung cancers, metastasize quickly to multiple sites, whereas others, such as breast and prostate carcinomas, often take years to develop metastatic colonies and then only in a relatively limited number of sites \cite{hess2006metastatic}. Importantly, metastatic colonies at different anatomical sites may {respond} differently to anti-cancer drugs or have drastically different growth kinetics. 

Another shortcoming of the previous work \cite{badri2015minimizing} was that the mathmatical model ignored the growth of tumors at metastatic sites. In particular, \cite{badri2015minimizing} only looked at \textit{metastatic risk}, which was measured by the total number of successful metastatic cells produced and thus treated all successful metastatic cells equally, regardless of what point in time they initiated metastatic lesions. However, a more accurate model of the clinical situation allows for the metastatic cells to reproduce in their new location \cite{chaffer2011perspective}. Incorporating this important phenomenon now means that metastatic sites that are initiated earlier are much more dangerous than metastatic lesions initiated late in the course of the disease. 

Third, our prior work \cite{badri2015minimizing} focused solely on radiotherapy. However, when studying the evolution of metastatic disease under treatment, the role of chemotherapy is clearly paramount, given its systemic nature. This has been born out in clinical trials where it has been observed that concurrent CRT reduces the risk of distant metastasis compared to radiotherapy alone (e.g., see \cite{wee2005randomized,lin2003phase}). In the current work, we address these shortcomings by developing a mathematical model that allows for multi-site metastatic disease with possible growth at each location and includes systemic chemotherapy as a possible treatment.

{\color{black}Finally, our previous work only considered static tumor radio-sensitivity parameters ($\alpha$ and $\beta$) of the LQ model over the course of radiotherapy. This ignores tumor re-oxygenation throughout the course of treatment, which is known to enhance tumor radio-sensitivity \cite{hall2006radiobiology,wein2000dynamic}. In this study, we also consider evolving radio-sensitivity parameters due to tumor oxygenation and study the resulting changes in optimal CRT regimens. This requires the development and implementation of a dynamic programming (DP) algorithm with a six-dimensional state space. A naive application of the state-discretization technique for this DP algorithm will be computationally intractable. Hence, we develop approximation methods using efficient data structures to significantly lower the required space to store state-space information. Additionally, we derive and use two mathematical results on the structure of the optimal solution to reduce the computational complexity of our DP algorithm.}

In this work, we consider the problem of finding optimal radiotherapy dosing and a chemotherapy regimen that minimizes the total expected number of metastatic tumor cells at multiple body sites while keeping {normal-tissue damage} below clinically acceptable levels. To study the trade-off between the conventional objective, maximizing TCP, and our novel approach of minimizing the metastatic population, we consider an additional constraint requiring the {\color{black}tumor control} associated with the optimal regimen to exceed a user-specified percentage of that of conventional regimens. To the best of our knowledge, the present study is the first to address optimal CRT fractionation in the context of targeting multi-site metastatic disease and primary tumor control. We examine the mathematical properties of the optimal fractionation scheme in the context of cervical tumors.

The remainder of this paper is organized as follows. In Section \ref{sec: model},  we discuss a model with static radio-sensitivity parameters for multi-site metastasis production and how it can be used to develop a function that reflects metastatic production at different body locations. In Section \ref{sec: solution}, we formulate an optimization model for the CRT fractionation problem, derive some mathematical results regarding the properties of optimal regimens {\color{black}with constant $\alpha$ and $\beta$, extend our model to incorporate dynamics of reoxygenation throughout the course of treatment, and develop a DP solution approach}. In Section \ref{sec: Results}, we present numerical results for a situation involving cervical cancer. Finally, we summarize and conclude the paper in Section \ref{sec: conclusion}. 

\section{Mathematical model of treatment and metastasis development}\label{sec: model}
Consider a treatment regimen in which a radiation dose of $d_i$ and a drug concentration of $c_i$ are administered at treatment fraction $i=1,\dots,N$. CRT treatment regimens are broadly divided into three classes: neoadjuvant, concurrent, or adjuvant, depending on whether the chemotherapeutic agents are administered before, during, or after the course of radiotherapy, respectively. We consider a sufficiently large $N$ to accommodate all possible regimens. {\color{black}The choice of sufficiently large $N$ is consistent with CRT clinical practice, where the duration of the therapy ranges between 6 to 10 weeks. Note that if $N$ is small, then the solution space is limited to concurrent CRT regimens with overlapping drug and radiation administration schedules, potentially excluding promising sequential CRT regimens.}

Our objective is to minimize total metastatic cell population produced by the primary tumor at $n$ different sites within the patient's body at time $T$ after the conclusion of the treatment, while limiting the normal-tissue toxicity and ensuring a certain level of control over the primary tumor. The time horizon $[0,T]$ consists of treatment days and a potentially long period of no treatment, at the end of which total metastatic population is evaluated. In the following sections, we first model the dynamics of the primary tumor and metastatic-cell populations in the presence of chemoradiation cell-kill. Then we discuss a clinically established model used to quantify the primary tumor control and normal-tissue toxicity. These models will be used to formulate the CRT fractionation problem. 

{\color{black}
\subsection{Modeling primary tumor population growth} \label{sec: tumorPop}
We start by modeling the population of primary tumor cells at time $t$ as a function of radiation dose using the well-known LQ model, originated from clinical observations that survival curves of irradiated cell lines follow a linear-quadratic fit on a logarithmic scale \cite{hall2006radiobiology}. According to the LQ model, the surviving fraction ($\mathrm{SF}_{\mathrm{RT}}$) of tumor cells after a single exposure to a radiation dose of $d$ is
\[\mathrm{SF}_{\mathrm{RT}}(d)=\exp\left(-\alpha d-\beta d^2 \right)\]
where $\alpha$ and $\beta$ are tumor radio-sensitivity parameters.

To account for the cytotoxic effect of chemotherapeutic agents on the primary tumor population, we consider the two action mechanisms of \emph{additivity} and \emph{radio-sensitization} associated with CRT and incorporate them into the LQ model. Additivity refers to an added tumor cell-kill due to chemotherapeutic agents, assuming no interaction with radiation. Radio-sensitization, on the other hand, refers to an enhancement in the radiation-induced cell-kill, caused by chemotherapeutic agents \cite{seiwert2007}. Additivity is incorporated into the LQ model by adding a linear term of the chemotherapy drug's concentration, denoted by $c$, to the exponent of the LQ model. This is motivated by the observations that survival curves of cell lines that are exposed to chemotherapeutic agents follow an exponential form in terms of the drug's concentration \cite{skipper1964experimental,jones2014potential}. To accommodate radio-sensitization, we use the approach suggested in \cite{salari2013mathematical} in which a multiplicative term of the radiation dose and chemotherapy drug's concentration is added to the exponent of the LQ model. This is based on clinical observations that radio-sensitizers increase the linear term of the LQ model without causing a considerable change in the quadratic term \cite{dai2013radiosensitivity,franken2011radiosensitization}. Hence, the total fraction of surviving cells after incorporating additivity and radio-sensitization ($\mathrm{SF}_{\mathrm{CRT}}$) can be modified as
\[ \mathrm{SF}_{\mathrm{CRT}}(d,c)=\exp\left(-\alpha d-\beta d^2-\theta c -\psi d c\right) \]
where $\theta$ and $\psi$ are chemo-sensitivity parameters associated with the additivity and radio-sensitization terms, respectively.}

The choice of tumor-growth function is strongly driven by the size of cancer being modeled. Exponential growth is a good model for small-sized tumors. However, limitations of the availability of nutrients, oxygen, and space imply that exponential growth is not appropriate for the long-term growth of solid tumors, suggesting alternative formulations such as Gompertz or logistic models. The current view of tumor-growth kinetics during therapy is based on the general assumption that tumor cells grow exponentially due to the small population \cite{hall2006radiobiology}. Here, we also assume that the tumor is small enough and that it grows exponentially with a time lag during the course of the treatment, known as tumor \emph{kick-off time}. In particular, we are interested in the long-term metastatic risk associated with a CRT regimen evaluated at any time point between the conclusion of therapy and the local recurrence. During this time period, the primary tumor has not yet reached a clinically detectable size again and thus is sufficiently small enough to model with an exponential function.

Consider a CRT regimen $\left(d_i,c_i\right)\;(i=1,\ldots,N)$ that delivers a radiation dose of $d_i$, measured in Gy, and a chemotherapy dose of $c_i$, measured in milligrams per squared meter of body area, on day $i=1,\dots,N$. The primary-tumor population size at time $t\in[0,T]$, denoted by $X_t$, can then be computed by{\color{black}
\begin{equation}\label{xt}
X_t=\begin{cases}X_0e^{-\sum_{i=1}^{t}\left( \alpha d_i+\beta d_i^2+ \theta c_{i}+\psi d_i c_i\right) +\frac{\ln 2}{\tau_d}(t-T_k)^+},& t\le N\\
 X_0 e^{- \sum_{i=1}^{N}(\alpha d_i +\beta d_i^2+\theta c_{i}+\psi d_i c_i) +\frac{\ln 2}{\tau_d}(t-T_k)^+}, & t> N \end{cases}
\end{equation}\normalsize}
where $X_0$ is the initial tumor population, $T_k$ is the tumor kick-off time, and $\tau_d$ shows the tumor doubling time in units of days.
{\color{black}
\subsection{Modeling metastatic population growth}
We adopt a stochastic model used by other studies \cite{hanin2006stochastic,hanin2010,hanin2013} to describe the metastatic production and formation process. Here, we provide details of this model.} The model assumes that only a fraction of the primary tumor cells are capable of producing successful metastatic cells and that these cells do so at rate $\nu$. This fraction is dictated by the fractal dimension of the blood vessels infiltrating the tumor, which we denote by $\xi\in [0,1]$ \cite{iwata2000dynamical}. For example, if we assume that the tumor is roughly spherical and all cells on the surface are capable of metastasis, then we would take $\xi=2/3$; however, if we assume only the cells along the boundary of a one-dimensional blood vessel are capable of metastasis, then we would set $\xi=1/3$ (see \cite{hanin2006stochastic} for a more detailed discussion and estimation of this parameter). A non-homogeneous Poisson process with intensity function $\nu X_t^\xi$ is then used to describe the successful metastatic-cell production. Each successful metastatic cell, independent of other cells, will initiate a metastatic colony at site $i\;(i=1,\ldots,n)$ with probability $p_i$ {($\sum_{i=1}^n p_i=1$)}. Finally, it is assumed that the probability of secondary metastasis, which is the formation of new metastasis originating in other metastatic colonies, is negligible. Therefore, thinning the original Poisson process tells us that metastasis formation at site $i$ follows a non-homogeneous Poison process with rate
$b_i(t)= p_i\nu X_t^\xi.$

We model the growth of each metastatic lesion using a non-homogeneous branching process with a rate dependent on time $t$ and the metastasis site. More specifically, the birth and death rates at site $i\;(i=1,\ldots,n)$ are denoted by $u_i(t)$ and $v_i(t)$, respectively. This is similar to the modeling approach taken in \cite{haeno2012computational}, which used multi-type branching processes to study metastasis and then applied the work to study pancreatic cancer data. It should be noted that this work were able to successfully fit this model to patient data. We consider the growth rate of metastatic cells $\mu_i(t)=u_i(t)-v_i(t)$ in the absence of chemotherapeutic agents as a positive value that is constant with respect to time ($\mu_i(t)=\mu_i$). During chemotherapy, we assume that the growth rate of metastatic cells decreases linearly as a function of the magnitude of the last dose of chemotherapy drug administrated \cite{mould2015developing}. Written mathematically we have
\begin{equation}\label{metgrowth}
\mu_{i}(t)=\begin{cases}\mu_{i}-\omega_i c_{\lfloor t \rfloor} ,& t\le N\\
\mu_{i} ,& t\in (N,T]\end{cases},\ \  i=1,\dots,n
\end{equation}
where $\omega_i$ is a positive constant, depending on heterogeneity in drug concentrations among different sites and sensitivity of tumor cells to the chemotherapy at {\color{black}the $i$th site.}

If we have a total evaluation horizon of time $T$, then the expected number of total metastatic cells at site $i$ is given by the convolution (see \cite{foo2010evolution})
$$R_i(T)=\int_{0}^{T}b_i(t)\exp{\left(\int_{0}^{T-t}\mu_{i}(\eta+t) d\eta \right)dt }.$$
Our goal is to determine an optimal chemo- and radiotherapy fractionation scheme that minimizes the total expected number of metastatic tumor cells at time $T$ at $n$ body sites, i.e.,
\begin{equation}\label{eq_Rt}
R(T)=\sum_{i=1}^{n}R_i(T)=\nu\sum_{i=1}^{n}p_i\int_{0}^{T} X_t^\xi\exp{\left(\int_{0}^{T-t}\mu_{i}(\eta+t) d\eta  \right)dt }.
\end{equation}
Notice that in order to minimize $R(T)$, it is sufficient to optimize $R(T)/\nu$, and there is no need to know parameter $\nu$, which is difficult to estimate.
{\color{black}
\subsection{Modeling primary tumor control and normal-tissue toxicity due to radiation}}
The notion of \emph{biologically equivalent dose} (BED), originally motivated by the LQ model, is used in clinical practice to quantify the biological damage caused by a given radiotherapy regimen in a structure. More specifically, the BED for a radiotherapy regimen consisting of $N$ treatment fractions in which radiation dose $d_i$ is administered in fraction $i\;(i=1,\ldots,N)$ is given by
\begin{equation}\label{eq-beddef}{\color{black}
	\textrm{BED} = \sum_{i=1}^N d_i \left(1+\frac{d_i}{\left[\alpha/\beta \right]} \right) }
\end{equation}
where $\left[\alpha/\beta \right]$ is a structure-specific radio-sensitivity parameter. Several clinical studies have associated different BED values, independent of the number of fractions and the radiation dose per fraction $d_i\;(i=1,\ldots,n)$, to TCP and NTCP values calculated, based on clinical data \cite{fowler201421}. {\color{black}We will use BED to measure and compare the primary tumor control and the normal-tissue toxicity of the radiotherapy schedule $d_i\;(i=1,\ldots,N$) in different CRT treatments.}

\section{Optimization model and solution approach} \label{sec: solution}
Our goal is to determine the optimal radiation dose vector $\vec{d^*}=\{d_1^*,\dots,d_N^*\}$ and optimal chemotherapy drug concentration vector $\vec{c^*}=\{c_1^*,\dots,c_N^*\}$ that minimize the total expected number of metastatic cancer cells at time $T$ at $n$ different anatomical sites. We do not consider the metastatic dissemination prior to the start of the treatment course. One reason for this is that it is difficult to estimate the total life history of a tumor prior to diagnosis. In addition, our model is primarily intended for locally advanced cancer cases, whereas metastasis formation typically occurs in later stages of the disease once the tumor has reached a significantly large size \cite{hanahan2000hallmarks}. 

We start by computing
$$
\int_{0}^{T-t}\mu_{i}(\eta+t) d\eta=\int_{t}^{T}\mu_{i}(s) ds=\begin{cases}\mu_i\cdot(T-t)-\omega_i\left(c_{\lfloor t \rfloor}(\lceil t\rceil-t )+ \sum_{j=\lceil t \rceil}^{N}c_{j}\right)  ,& t< N+1\\
\mu_i\cdot(T-t) ,& t\in [N+1,T]\end{cases}
$$
with notation $c_0=0$ and $\sum_{j=\lceil t \rceil}^{N}c_{j}=0$ for $t>N$. For computational ease, we assume an equal metastasis production for $t\in [ \lfloor t \rfloor,\lceil t \rceil )$, which is evaluated at time $\lfloor t \rfloor$. Specifically we replace ${X_t}$ with $X_t^-$ and $\mu_i\cdot(T-t)-\omega_i\left(c_{\lfloor t \rfloor}(\lceil t\rceil-t )+ \sum_{j=\lceil t \rceil}^{N}c_{j}\right)$ with $\mu_i\cdot(T-t)-\omega_i \sum_{j=\lfloor t \rfloor }^{N}c_{j}$ for $0\le t<N+1$, where $X_t^{-}$ is the number of cells immediately before the delivery of doses $d_t$ and $c_t$ (note that $X_0^{-}=X_0$). In addition, we use a similar approximation approach as that taken in \cite{badri2015minimizing}, which uses a summation rather than an integral during therapy.  To summarize, we make the following approximation:{\color{black}
\begin{equation}
R(T)/\nu\approx \sum_{i=1}^{n}p_i\left( \sum_{t=0}^{N+1}(X_t^-)^\xi\exp{\left(\mu_i\cdot(T-t)-\omega_i\sum_{j=t}^{T}c_{j} \right)}+\int_{N+1}^{T} X_t^\xi e^{\mu_i\cdot(T-t) }dt \right).
\label{eq-RTapprox}
\end{equation}
Substituting $X_t$ from \eqref{xt} into \eqref{eq-RTapprox} yields
$$
R(T)/\nu\approx f(\vec{d},\vec{c})+g\left( \sum_{t=1}^{N}d_t,\sum_{t=1}^{N}d_t^2,\sum_{t=1}^{N}c_t,\sum_{t=1}^{N}d_tc_t\right)
$$
where
\begin{equation}\label{ffunction}
f(\vec{d},\vec{c})=X_0^\xi\sum_{i=1}^{n}p_ie^{\mu_iT} \sum_{t=0}^{N+1}e^{-\xi\left( \sum_{j=1}^{t-1}(\alpha d_j+\beta d_j^2+\theta c_j+\psi d_i c_i)-\frac{\ln 2}{\tau_d}(t-T_k)^+\right)-\mu_it-\omega_i\sum_{j=t}^{T}c_j}
\end{equation}
and

\begin{equation}\label{gfunction}
g\left( \sum_{t=1}^{N}d_t,\sum_{t=1}^{N}d_t^2,\sum_{t=1}^{N}c_t,\sum_{t=1}^{N}d_tc_t\right)=X_0^\xi e^{-\xi\Gamma}\sum_{i=1}^{n}p_ie^{\mu_iT}\frac{e^{-(\mu_i-\xi\frac{\ln 2}{\tau_d})(N+1)}-e^{-(\mu_i-\xi\frac{\ln 2}{\tau_d})T}}{\mu_i-\xi\frac{\ln 2}{\tau_d}}
\end{equation}\normalsize
where
$$
\Gamma=\sum_{t=1}^{N} (\alpha d_t+\beta  d_t^2+\theta c_{t}+\psi d_tc_t)+\frac{\ln 2}{\tau_d}T_k.
$$
}Note that in \eqref{ffunction}, the summation inside the exponential for $t=0, 1$ is zero, i.e.,
$${\color{black}
\sum_{j=1}^{t-1}(\alpha d_j+\beta d_j^2+\theta c_j+\psi d_jc_j)=0}
$$
for $t=0,1$, and we assume $T_k<N$, which is typically the case for all disease sites. We now consider the fractionation problem in terms of decision variables of radiation dose $d_i$ and drug concentration $c_i$ at fraction $i=1,\dots,N$, which leads to a minimal metastasis population size while maintaining acceptable levels of normal tissues damage.
{\color{black}
\begin{equation}\label{eq-obj1}
\underset{d_t,c_t\ge0}{\text{minimize }} f(\vec{d},\vec{c})+g\left( \sum_{t=1}^{N}d_t,\sum_{t=1}^{N}d_t^2,\sum_{t=1}^{N}c_t,\sum_{t=1}^{N}d_tc_t\right)
\end{equation}}
\normalsize
s.t.
\begin{equation}\label{eq-bedcon1}
\sum_{t=1}^{N}\gamma_j d_t\left(1+\frac{\gamma_j d_t}{\left[\alpha/\beta\right]_j}\right)\le \textrm{BED}_j , \ \ j=1,\dots,M
\end{equation}
\begin{equation}\label{eq-bedtum}
\sum_{t=1}^{N} d_t\left(1+\frac{d_t}{\left[\alpha/\beta\right]}\right)\ge {\color{black}\varrho} \times\textrm{BED}_{\text{std}}
\end{equation}
\begin{equation}\label{eq-chemocon1}
\sum_{t=1}^{N}c_{t}\le C_{\max}
\end{equation}
\begin{equation}\label{eq-radcon1}
d_t\le d_{\max}, \ \ t=1,\dots,N
\end{equation}
\begin{equation}\label{eq-checon1}
c_{t}\le c_{\max}, \ \ t=1,\dots,N
\end{equation}

The objective function evaluates the total metastatic population (divided by $\nu$) at time $T$ at $n$ different body organs. Constraints \eqref{eq-bedcon1} and \eqref{eq-chemocon1} are toxicity constraints associated with radiotherapy and chemotherapy, respectively. More specifically, we use the BED model in (\ref{eq-beddef}) to constrain the radiation side effects in the surrounding healthy structures around the tumor \cite{hall2006radiobiology}. In doing so, we assume that there are $M$ healthy tissues in the vicinity of the tumor, and a dose $d$ results in a homogeneous dose $\gamma_jd$ in the $j$th normal tissue, where $\gamma_j$ is the sparing factor. Parameter $\left[\alpha/\beta\right]_j$ represents the BED parameter, and $\textrm{BED}_j$ specifies the maximum BED allowed in the $j$th healthy tissue. We limit the maximum cumulative chemotherapy dose by $C_{\max}$, which is the total drug dose that can be safely delivered to a patient over the period of treatment, measured in milligrams/meter squared \cite{mccall2008evolutionary}.  Inequalities \eqref{eq-radcon1} and \eqref{eq-checon1} limit the maximum allowable daily radiation dose and chemotherapy drug concentration, respectively, where $d_{\max}$ and $c_{\max}$ show the maximum tolerable dose, measured in Gy and milligrams/meter squared, respectively, that can be delivered in a single day \cite{mccall2008evolutionary}. 

Finally, constraint \eqref{eq-bedtum} guarantees that tumor BED imposed by the optimal schedule is greater than or equal to {\color{black}$\varrho\times 100$} percent of the tumor BED realized by the standard schedule. Note that TCP is an increasing function of tumor BED for a homogeneous radiation dose, and thus TCP and tumor BED maximization are equivalent.  {\color{black}The parameter $0\le \varrho \le 1$} is an input to our model and characterizes the trade-off between maximizing TCP and minimizing metastatic cell population (see, e.g., \cite{chankong1983multiobjective} for a review of optimization methods to characterize the trade-off between the two objectives). If the primary cause of cancer-related death in a specific tumor is metastatic dissemination (e.g. pancreatic or cervical cancer), then a relatively small value of {\color{black}$\varrho$} (e.g., $70\%$) may be used. On the other hand, if the primary tumor may cause death (e.g., lung cancer), then a larger value of {\color{black}$\varrho$} will be used (e.g., $98\%$).  

The formulation \eqref{eq-obj1}--\eqref{eq-checon1} is a nonconvex, quadratically constrained program. Such problems are computationally difficult to solve in general. In the rest of this section, we first derive some mathematical results regarding the structure of the optimal solution of \eqref{eq-obj1}--\eqref{eq-checon1} and then develop a DP algorithm to solve the problem.

We start by discussing an important property of the optimal chemotherapy dose vector.
\begin{lemma}\label{lemmasum}
Let the optimal drug vector of problem \eqref{eq-obj1}--\eqref{eq-checon1} be $\vec{c^*}=\{c_1^*,\dots,c_{N}^*\}$; then in optimality, we have $$\sum_{t=1}^{N}c_t^*=C_{\max}.$$
\end{lemma}

The proof is in Appendix \ref{lemmasumProof}. Lemma \ref{lemmasum} states that any optimal solution to \eqref{eq-obj1}--\eqref{eq-checon1} yields a chemotherapy regimen with the maximum dose allowed. Using the result of Lemma \ref{lemmasum}, we can rewrite function $f(\vec{d},\vec{c})$ as {\color{black}
\begin{equation}\label{eq-revisedf}
f(\vec{d},\vec{c})=X_0^\xi \sum_{i=1}^{n}p_ie^{\mu_iT}\sum_{t=0}^{N+1}e^{-\xi\left( \sum_{j=1}^{t-1}(\alpha d_j+\beta d_j^2+(\theta-\omega_i/\xi) c_j+\psi d_jc_j)-\frac{\ln 2}{\tau_d}(t-T_k)^+\right)-\mu_it-\omega_iC_{\max}}.
\end{equation}

\subsection{Developing DP solution method} \label{sec: DP}
Using the new form of function $f(\vec{d},\vec{c})$ derived in \eqref{eq-revisedf}, we develop a DP algorithm to find the optimal daily radiation and chemotherapy doses. Stages of the DP algorithm correspond to treatment days $t \in \left\{1,\ldots,T\right\}$. At stage $t$, the state of the system is characterized by $\mathcal{A}_t=\left\{U_t,V_t,S_t,W_t\right\}$, where $U_t$ and $V_t$ are the cumulative dose and cumulative dose squared, respectively, $S_t$ is the cumulative chemotherapy dose, and $W_t$ is the cumulative product of daily radiotherapy and chemotherapy doses, all calculated after the delivery of radiation and chemotherapy on day $t$. The control variables at stage $t$ are the radiation and chemotherapy doses administered during that stage. Given control variables $\left(d_t,c_t\right)$, the state of the system is updated according to
\[U_{t}=U_{t-1}+d_{t},\ \ V_{t}=V_{t-1}+d_{t}^2,\ \ S_{t}=S_{t-1}+c_{t}, \ \ W_t=W_{t-1}+d_tc_{t}.\]
Let $J_t: \mathbb{R}^4 \rightarrow \mathbb{R}$ be the cost-to-go function at stage $t \;(t=1,\ldots,T)$. Using the functions introduced in \eqref{gfunction} and \eqref{eq-revisedf}, the forward recursion of our DP algorithm can be written as

$$J_{t}(\mathcal{A}_t)=\begin{cases}\min_{d_{t},c_t\ge0}\ \ J_{t-1}(\mathcal{A}_{t-1})+L(d_{t},c_t,\mathcal{A}_{t-1}) ,& 1\le t\le N-1\\
\min_{d_{t},c_t\ge0} \ \ J_{t-1}(\mathcal{A}_{t-1})+g(U_{t-1}+d_t,V_{t-1}+d_t^2,C_{\max},W_{t-1}+d_tc_t) ,& t=N\end{cases}$$
\normalsize
where
$$
L(d_{t},c_t,\mathcal{A}_{t-1})=X_0^\xi \sum_{i=1}^{n}p_i e^{-\xi\times\Psi_{t,i}+\mu_i\cdot(T-t)-\omega_i C_{\max}}
$$\normalsize
$$
\Psi_{t,i}= \alpha  (U_{t-1}+d_{t})+\beta (V_{t-1}+d_{t}^2)+(\theta-\omega_i/\xi) (S_{t-1}+c_t)+\psi(W_{t-1}+d_tc_t)-\frac{\ln 2}{\tau_d}(t-T_k)^+
$$
with initial conditions $\mathcal{A}_{0}=\{0,0,0,0\}$ and $J_0(\mathcal{A}_{0})=X_0^\xi \sum_{i=1}^{n}p_ie^{\mu_i T-\omega_i C_{\max}}.$ We set the function $J_N(\mathcal{A}_{N})$ to be $\infty$ if  
$$ 
U_N+\frac{\beta}{\alpha} V_N< \varrho\times \textrm{BED}_{\text{std}}
$$ 
and $J_t(\mathcal{A}_{t})$, $t=1,\dots,N$ to be $\infty$, if any of the following constraints is violated:
$$
{\gamma_j U_t+\frac{\gamma_j^2 V_t}{\left[\alpha/\beta\right]_j}> \textrm{BED}_j, \ \ j=1,\dots,M; \text{ or } S_t>C_{\max}.}
$$

A standard approach to solving Bellman's equations with forward recursion is to discretize the state space and use a linear interpolation of appropriate discretized values to increase the accuracy of cost-to-go function estimations at intermediate state values \cite{badri2015minimizing,salari2013mathematical,bortfeld2013optimization}. Depending on the size of the state space, a naive implementation of this approach may be computationally intractable due to the large number of possible discretized states, known as the curse of dimensionality. To alleviate the computational burden of the discretization approach in our DP framework, we employ a dynamic table data structure to only store the information related to feasible states. Columns and rows of this table correspond to state attributes and feasible states at different DP stages, respectively. This leads to a significant reduction in the number of discretized states considered. The details of this implementation is in Appendix \ref{algorithm}.

In the rest of this section, we study two special cases of the proposed model. The first special case considers CRT regimens in the absence of radio-sensitization and derives a closed-form solution for the optimal chemotherapy schedule, which significantly reduces the computational complexity of the DP algorithm. The second case considers a new version of the proposed model in which tumor radio-sensitivity parameters vary during the course of the treatment.}
{\color{black}
\subsection{Optimal CRT regimens in the absence of radio-sensitization}
Several cytotoxic agents, such as temozolomide, are suggested to only have an additive effect
\cite{chalmers2009cytotoxic}. This corresponds to those chemotherapeutic agents that exhibit insignificant or no radio-sensitization, which translates to $\psi=0$ in tumor population dynamics introduced in \eqref{xt}. In this section, we solve the problem \eqref{eq-obj1}--\eqref{eq-checon1} for the special case of negligible radio-sensitization effect ($\psi=0$).} First, we use the new form of function $f(\vec{d},\vec{c})$ stated in \eqref{eq-revisedf} and show an interesting property of the optimal structure of radiation and chemotherapy schedules.
\begin{lemma}\label{lemmaordering}
Let the optimal dose and drug vector of problem \eqref{eq-obj1}--\eqref{eq-checon1} be $\vec{d^*}=\{d_1^*,\dots,d_{N}^*\}$ and $\vec{c^*}=\{c_1^*,\dots,c_{N}^*\}$, respectively; {\color{black}if $\psi=0$,} then, in optimality, we have $$d_1^*\ge d_2^*\ge\dots\ge d_{N}^*$$ and 
$$\begin{cases}
c_1^*\ge c_2^*\ge\dots\ge c_{N}^*, & \text{ if }\theta\xi\ge\max_{i=1,\dots,n} \omega_i\\
c_1^*\le c_2^*\le\dots\le c_{N}^*, & \text{ if }\theta\xi<\min_{i=1,\dots,n} \omega_i.
\end{cases}$$
\end{lemma} 

The proof is in Appendix \ref{lemmaorderingProof}. In Lemma \ref{lemmaordering}, we specify the general structure of the optimal CRT regimen in the absence of any radio-sensitivity effect. In the following result, we find the optimal chemotherapy drug vector for two special cases.
\begin{theorem}\label{theoremchemo}
Let the optimal drug vector of problem \eqref{eq-obj1}--\eqref{eq-checon1} be $\vec{c^*}=\{c_1^*,\dots,c_{N}^*\}$; {\color{black}if $\psi=0$;} then  $\vec{c^*}$ takes one of the following two forms, if we have $\theta\xi\ge\max_{i=1,\dots,n} \omega_i$ or $\theta\xi<\min_{i=1,\dots,n} \omega_i$:
\begin{enumerate}
\item $\theta\xi\ge\max_{i=1,\dots,n} \omega_i$: Optimal schedule is given by $c_i^*=c_{\max}$ for $i=1,\dots,k$, $c_{k+1}^*=C_{\max}-kc_{\max}$ and $c_i^*=0$ for $i=k+2,\dots,N$, where $k=\lfloor C_{\max}/c_{\max}\rfloor$.
\item $\theta\xi<\min_{i=1,\dots,n} \omega_i$: Optimal schedule is given by $c_i^*=c_{\max}$ for $i=N-k+1,\dots,N$, $c_{N-k}^*=C_{\max}-kc_{\max}$ and $c_i^*=0$ for $i=1,\dots,N-k-1$, where $k=\lfloor C_{\max}/c_{\max}\rfloor$.
\end{enumerate}
\end{theorem}
The proof is in Appendix \ref{theoremchemoProof}. The following three important observations are made for optimal CRT treatment regimens based on the results of Lemma \ref{lemmaordering} and Theorem \ref{theoremchemo}. {\color{black}Note that in the following, we assume that radio-sensitization effect is negligible.}

\begin{remark}
The non-increasing order of the optimal radiation dose vector obtained in Lemma \ref{lemmaordering} suggests that it is always optimal to immediately start the radiotherapy treatment. Furthermore, it is optimal to immediately initiate chemotherapy if $\theta\xi\ge\max_{i=1,\dots,n} \omega_i$, or to postpone chemotherapy until the end of the treatment if $\theta\xi<\min_{i=1,\dots,n} \omega_i$,  suggesting a concurrent and adjuvant CRT regimen, respectively. Hence, Lemma \ref{lemmaordering} tells us that it is never optimal to use neoadjuvant CRT regimens, which is consistent with clinical observations \cite{green2001survival}(see conclusion for more discussion).
\end{remark}

\begin{remark}
The resulting radiotherapy fractionation schedules for minimizing total metastatic cancer cells suggest a non-increasing radiation fractionated structure, which is due to the structure of the model. Metastatic cells initiated at distant locations are not killed by radiotherapy doses. Thus, in order to minimize the metastatic cell population using radiation, it is necessary to quickly reduce the primary tumor cell population since this is the source of metastatic lesions. Also, note that given the growth of metastatic lesions, we are particularly concerned with metastatic lesions created early in the course of therapy.
\end{remark}

\begin{remark}
The resulting chemotherapy fractionation schedules for minimizing total metastatic cancer cells suggest a hypo-fractionated structure concentrated at the beginning or the end of the schedule. This is due to the fact that if the combined effect $(\theta\xi)$ of chemotherapy-induced cell-kill $(\theta)$ and the fraction of cells capable of metastasis in the primary tumor $(\xi)$ are larger than the chemotherapy-induced metastatic cell-kill at all distant locations $(\max_{i=1,\dots,n}\{\omega_i\})$, then we consider the primary tumor cell population as dangerous, due to the inability of chemotherapy to control the metastatic cell populations. Hence, it is preferable to target the primary tumor and reduce its size as quickly as possible. However, if the effect of the chemotherapeutic agents on all distant metastatic sites $(\min_{i=1,\dots,n}\{\omega_i\})$ is larger than the product of chemotherapy-induced cell-kill $(\theta)$ and the fraction of cells that are capable of metastasis in the primary tumor $(\xi)$, then we consider that the chemotherapeutic agents are more effective in controlling metastatic disease than the primary tumor control. Hence, we prefer to administer chemotherapy at the end of treatment when we can target the accumulated metastatic cells.
\end{remark}

Finally, we establish a result that describes the closed-form solution of the optimal radiotherapy schedule for a special case in which $\beta=0$ in equation (\ref{xt}), $\left[\alpha/\beta \right]_j \rightarrow \infty\;(j=1,\ldots,M)$ in inequalities (\ref{eq-bedcon1}), and $\left[\alpha/\beta \right] \rightarrow \infty$ in inequality (\ref{eq-bedtum}). The special case entails the use of a linear log-survival model, rather than LQ, to describe tumor cell-kill. Also, it implies the use of physical dose, rather than BED, to measure normal-tissue toxicity and primary tumor control. This case is motivated by the fact that tissue response to small radiation doses is dominated by the linear term \cite{sachs2005solid,awwad2013radiation}, which is applicable to our model when $d_{\max}$ is sufficiently small (around $2$ Gy \cite{nahum1992maximizing}).

\begin{proposition}
If we assume that radiation-induced tumor cell-kill has a linear form, i.e., $\beta=0$ and radiation-induced toxicity in normal-tissue $j\;(j=1,\ldots,M)$ and the primary tumor control are evaluated using the physical dose, i.e., $\left[\alpha/\beta\right]_j \rightarrow \infty$ and $\left[\alpha/\beta \right] \rightarrow \infty$, then the optimal radiotherapy schedule is given by $d_i^*=d_{\max}$ for $i=1,\dots,k$, $d_{k+1}^*=D_{\max} -kd_{\max}$ and $d_i^*=0$ for $i=k+2,\dots,N$, where $k=\lfloor D_{\max} /d_{\max}\rfloor$ and $D_{\max}=\min_{j=1,\dots,M}\left\lbrace \normalfont{\textrm{BED}}_j/\gamma_j \right\rbrace$.
\end{proposition}
In order to explain this result, we first note that the inequality $\sum_{t=1}^{N}d_t\le \min_{j=1,\dots,M}\left\lbrace\textrm{BED}_j/\gamma_j  \right\rbrace $ defines the set of feasible regions associated with the radiation-induced toxicity constraints in normal tissues, and thus we can remove the remaining $M-1$ redundant inequalities from the constraint set. Then, an almost identical argument as used in proof of {\color{black}Theorem \ref{theoremchemo}} for the case of $\theta\xi\ge\max_{i=1,\dots,n}\omega_i$ can be implemented to derive the optimal dose vector.

{\color{black}If either of the conditions outlined in Theorem \ref{theoremchemo} are applicable, then the optimal chemotherapy schedule exists in closed form. It then remains to find the optimal radiation doses, which can be determined using a reduced version of the DP algorithm developed in Section \ref{sec: DP} and considering only a 2D state space consisting of $U_t$ and $V_t$. Otherwise, the optimal solution can be obtained using the DP algorithm with a 3D state space consisting of $U_t$, $V_t$, and $S_t$, since there is no radio-sensitization in this case.}
{\color{black}
\subsection{Dynamic tumor radio-sensitivity parameters}
The degree of tumor oxygenation is a crucial determinant of the effectiveness of radiotherapy. In particular, an insufficient supply of oxygen in tumor, known as \emph{hypoxia}, causes radio-resistance in most tumors, thereby  adversely affecting the cancer treatment outcome in solid tumors \cite{wouters1997cells}. It is suggested that radiation fractionation alleviates tumor hypoxia through the phenomenon of re-oxygenation, the process by which hypoxic cells surviving a given radiation fraction become more oxygenated and thus more susceptible to radiation prior to the next fraction \cite{hall2006radiobiology}. In this section, we account for the tumor re-oxygenation effect by explicitly modeling the dependence of tumor radio-sensitivity parameters on the tumor oxygen level. To this end, we first review a mathematical model for calculating oxygen-dependent radio-sensitivity parameters. We then extend the tumor population dynamics introduced in Section \ref{sec: tumorPop} to the case of dynamic radio-sensitivity parameters. Finally, we propose an efficient DP algorithm to solve the revised model.

The common approach to modeling the evolution of radio-sensitivity parameters is to incorporate the concept of the \emph{oxygen enhancement ratio} (OER), which models the dependence of radio-sensitivity parameters on tumor oxygen partial pressure as follows \cite{wouters1997cells,saberian2015theoretical}:
\begin{equation}\label{eq-OER}
\alpha_t=\frac{\alpha_{\max}}{\text{OER}_\alpha}\left(\frac{y_t\times \text{OER}_\alpha+K}{y_t+K} \right), \ \ \beta_t=\frac{\beta_{\max}}{\text{OER}_\beta^2}\left(\frac{y_t\times \text{OER}_\beta+K}{y_t+K} \right) ^2
\end{equation}
In the equations above, $\alpha_{\max}$ and $\beta_{\max}$ are radio-sensitivity parameters under well-oxygenated conditions, and $y_t$ denotes the average tumor oxygen partial pressure at time $t$. Additionally, $\text{OER}_\alpha$, $\text{OER}_\beta$, and $K$ are the fitted parameters of the OER model. To model the evolution of tumor oxygen partial pressure, we consider a spherical tumor, the radius of which, denoted by $r_t$, changes throughout the treatment course as a result of radiation and chemotherapy cell-kill as well as tumor repopulation. In particular, assuming a constant tumor-cell density per unit volume, denoted by $\rho$, we express the radius of the tumor prior to the start of treatment session $t$ as a function of the tumor population size, that is, $r_t = \sqrt[3]{\frac{3X_{t-1}}{4\pi\rho}}$ \cite{wein2000dynamic}. We then assume a simple model in which the average oxygen partial pressure $y_t$ in \eqref{eq-OER} is linearly proportional to the radius of the spherical tumor, that is, $y_t\propto r_t$. This assumption is motivated by clinical studies indicating that large tumors have large hypoxic regions \cite{de1998heterogeneity}. Substituting $X_{t-1}$ from equation \eqref{xt}, we can express the average tumor oxygen partial pressure on treatment session $t$ as 
\begin{equation}\label{eq-oxpress}
y_t= y_{\max}-\iota\times \sqrt[3]{\frac{3X_0}{4\pi\rho}}\exp\left(-\frac{1}{3}\sum_{i=1}^{t-1}\left( \alpha_id_i+\beta_id_i^2+\theta c_i+\psi d_ic_i\right)+\frac{\ln 2}{3\tau_d}(t-T_k)^+ \right) 
\end{equation}
where $y_{\max}$ is the maximum average oxygen partial pressure, usually achieved toward the end of treatment when the tumor radius is sufficiently small, and $\iota$ is the model parameter representing the increase in average oxygen partial pressure per unit reduction in tumor radius. Note that the value $ y_{\max}-\iota\times \sqrt[3]{\frac{3X_0}{4\pi\rho}}$ represents the average oxygen partial pressure for the tumor at the beginning of the therapy, which depends on the initial tumor size and the tumor oxygenation level. Hence, we can use the average tumor oxygenation level of an arbitrary tumor at the beginning and end of treatment to calibrate our model.

Equation (\ref{xt}) describing the tumor population size needs to be updated to account for dynamic tumor radio-sensitivity parameters as follows:
\[X_t=\begin{cases}X_0e^{-\sum_{i=1}^{t}\left( \alpha_i d_i+\beta_i d_i^2+ \theta c_{i}+\psi d_i c_i\right) +\frac{\ln 2}{\tau_d}(t-T_k)^+},& t\le N\\
 X_0 e^{- \sum_{i=1}^{N}(\alpha_i d_i +\beta_i d_i^2+\theta c_{i}+\psi d_i c_i) +\frac{\ln 2}{\tau_d}(t-T_k)^+}, & t> N. \end{cases}\]
This changes the objective function of problem \eqref{eq-obj1}--\eqref{eq-checon1} to 
$$
\hat{f}(\vec{d},\vec{c},\vec{\alpha},\vec{\beta})+\hat{g}\left( \sum_{t=1}^{N}\alpha_td_t,\sum_{t=1}^{N}\beta_td_t^2,\sum_{t=1}^{N}c_t,\sum_{t=1}^{N}d_tc_t\right)
$$\normalsize
where
$$\normalsize
\hat{f}(\vec{d},\vec{c},\vec{\alpha},\vec{\beta})=X_0^\xi \sum_{i=1}^{n}p_ie^{\mu_iT}\sum_{t=0}^{N+1}e^{-\xi\left( \sum_{j=1}^{t-1}(\alpha_j d_j+\beta_j d_j^2+(\theta-\omega_i/\xi) c_j+\psi d_jc_j)-\frac{\ln 2}{\tau_d}(t-T_k)^+\right)-\mu_it-\omega_iC_{\max}}
$$
$$
\hat{g}=\left( \sum_{t=1}^{N}\alpha_td_t,\sum_{t=1}^{N}\beta_td_t^2,\sum_{t=1}^{N}c_t,\sum_{t=1}^{N}d_tc_t\right)=X_0^\xi e^{-\xi\hat{\Gamma}}\sum_{i=1}^{n}p_ie^{\mu_iT}\frac{e^{-(\mu_i-\xi\frac{\ln 2}{\tau_d})(N+1)}-e^{-(\mu_i-\xi\frac{\ln 2}{\tau_d})T}}{\mu_i-\xi\frac{\ln 2}{\tau_d}}
$$
$$
\hat{\Gamma}=\sum_{t=1}^{N} (\alpha_t d_t+\beta_t  d_t^2+\theta c_{t}+\psi d_tc_t)+\frac{\ln 2}{\tau_d}T_k.
$$
To solve the formulation in \eqref{eq-obj1}--\eqref{eq-checon1} with the modified objective function using the DP approach, we need to store two additional state variables, $\hat{U}_t=\sum_{i=1}^{t-1}\alpha_td_t$ and $\hat{V}_t=\sum_{i=1}^{t-1}\beta_td_t^2$. This leads to a 6D state space characterized by $\mathcal{B}_t=\{U_t,V_t,S_t,W_t,\hat{U}_t,\hat{V}_t\}$. The cost-to-go functions $\hat{J}_t: \mathbb{R}^6 \rightarrow \mathbb{R}$ for the new state space are defined as
$$\hat{J}_{t}(\mathcal{B}_t)=\begin{cases}\min_{d_{t},c_t\ge0}\ \ \hat{J}_{t-1}(\mathcal{B}_{t-1})+\hat{L}(d_{t},c_t,\mathcal{B}_{t-1}) ,& 1\le t\le N-1\\
\min_{d_{t},c_t\ge0} \ \ \hat{J}_{t-1}(\mathcal{B}_{t-1})+\hat{g}(\hat{U}_{t-1}+d_t,\hat{V}_{t-1}+d_t^2,C_{\max},W_{t-1}+d_tc_t) ,& t=N\end{cases}$$
\normalsize
where
$$
\hat{L}(d_{t},c_t,\mathcal{B}_{t-1})=X_0^\xi \sum_{i=1}^{n}p_i e^{-\xi\times\hat{\Psi}_{t,i}+\mu_i\cdot(T-t)-\omega_i C_{\max}}
$$\normalsize
$$
\hat{\Psi}_{t,i}= \hat{U}_{t-1}+\alpha_td_{t}+ \hat{V}_{t-1}+\beta_td_{t}^2+(\theta-\omega_i/\xi) (S_{t-1}+c_t)+\psi(W_{t-1}+d_tc_t)-\frac{\ln 2}{\tau_d}(t-T_k)^+.
$$
Tumor radio-sensitivity parameters $\alpha_t$ and $\beta_t$ $(t=1,\ldots,N)$ used in the cost-to-go functions can be computed using equations \eqref{eq-OER}--\eqref{eq-oxpress}.
The dynamic table data structure employed in our DP implementation significantly reduces the computational and space requirements of our DP algorithm with a 6D state space. Details of the DP algorithm and implementation for the 6D state space are provided in Appendix \ref{algorithm-dynamic}.
}
\section{Numerical results} \label{sec: Results}
In this section, we test the performance of our formulation and solution methods on a locally advanced cervical cancer case. We consider the toxicity effects of radiotherapy in three different organs-at-risk(OAR): bladder, small intestine, and rectum \cite{portelance2001intensity}, using the BED model introduced in (\ref{eq-beddef}). A standard fractionated treatment for cervical cancer delivers $45$ Gy to the tumor over five weeks ($25$ working days). The maximum tolerated doses for each organ at risk is computed based on this standard fractionation scheme ($1.8$ Gy $\times$ 25). 

Treatment days comprise a course of external beam radiotherapy and chemotherapy administration, followed by a short rest period lasting $N_r$ days, and low-dose-rate brachytherapy with a dose rate of $R$ Gy per day delivered within $N_b$ days. This treatment sequence is motivated by the combined-modality treatment protocols used in clinical practice for cervical cancer, which often involve adjuvant brachytherapy. Equation \eqref{xt} can be adjusted to account for the additional brachytherapy as follows (see \cite{plataniotis2008use} for more details):{\color{black}
\begin{equation}\label{eq-xtadj}
X_t=\begin{cases}X_0e^{-\sum_{i=1}^{t}\left( \alpha  d_i+\beta  d_i^2+ \theta c_{i}+\psi d_ic_i\right) +\frac{\ln 2}{\tau_d}(t-T_k)^+},& t\le N+N_r\\
X_0 e^{-\sum_{i=1}^{N}\left( \alpha  d_i+\beta  d_i^2+ \theta c_{i}+\psi d_ic_i\right)-e_gR(t-N-N_r)(\alpha_N +2\beta_N R/\sigma) +\frac{\ln 2}{\tau_d}(t-T_k)^+}, & N+N_r<t\le N+N_r+N_b \\ X_0 e^{- \sum_{i=1}^{N}(\alpha d_i +\beta d_i^2+\theta c_{i}+\psi d_ic_i)-e_gRN_b(\alpha_N +2\beta_N R/\sigma) +\frac{\ln 2}{\tau_d}(t-T_k)^+}, & t> N+N_r+N_b \end{cases}
\end{equation}\normalsize}
where  $\sigma$ is the monoexponential repair rate of sublethal damage, and $e_g$ represents the effects of dose gradients around the brachytherapy
sources (see \cite{plataniotis2008use,dale1997calculation}). We assume that low-dose brachytherapy is performed in two days within one week after the completion of pelvic external radiotherapy \cite{plataniotis2008use}. Note that this assumption does not affect the optimal solution, and our result still applies to different brachytherapy schedules.  

The standard chemotherapy agents used in locally advanced cervical tumors are cisplatin, given at low daily doses, and fluorouracil (5-FU), given at high daily doses \cite{keys1999cisplatin,wee2005randomized}. We assume that cisplatin, when delivered at low doses, only augments the effects of radiotherapy through radio-sensitization without any independent cytotoxic effect \cite{britten1996effect,vokes1990concomitant}. Hence, the LQ parameter $\alpha$ in \eqref{eq-xtadj} is substituted with $\alpha+\kappa c_{\text{cis}}$, where $c_{\text{cis}}$ shows the average daily dose of cisplatin, as discussed in Section \ref{sec: tumorPop}. cisplatin is usually administered intravenously at a dose of 40 mg per square meter of body-surface area ($\textrm{mg}/\textrm{m}^2$) per week  \cite{keys1999cisplatin}; therefore, we set the daily dose $c_{\text{cis}}=8  \textrm{ mg}/\textrm{m}^2$. The agent 5-FU administrated at high doses is considered to have an inherent cytotoxic activity acting as a systemic therapeutic agent \cite{seiwert2007}.  The standard concurrent CRT protocol for cervical cancer consists of administrating a 5-FU dose of $1,000 \textrm{ mg}/\textrm{m}^2$ per day for 8 days \cite{peters2000concurrent}. Hence, we set the maximum daily dose of 5-FU to be $c_{max}=1,000 \textrm{ mg}/\textrm{m}^2$ and the total dose to be $C_{\max}=8,000\textrm{ mg}/\textrm{m}^2$.

We set $N$ to be 25 days (five weeks, considering weekends as breaks) since 5-FU is given concurrently with radiotherapy.
We employ the calibration method discussed in \cite{plataniotis2008use} to estimate the values of parameters $\kappa$, $\psi$, and $\theta$ associated with radio-sensitization of cisplatin and 5-FU and additivity of 5-FU, respectively, from clinical trial studies. {\color{black}In particular, to estimate parameter $\kappa$, we use the clinical trial study reported by \cite{keys1999cisplatin} to compare the progression-free survival rate of patients treated with radiotherapy alone, denoted by $\textrm{TCP}_{\textrm{\tiny{RT}}}$, against that of patients treated with radiotherapy plus cisplatin, denoted by $\textrm{TCP}_{\textrm{\tiny{CRT}}}$. This yields the following estimation for parameter $\kappa$ (see \cite{plataniotis2008use} for more details):
$$
\kappa = -\frac{\ln(\ln(\textrm{TCP}_{\textrm{\tiny{CRT}}})/\ln(\textrm{TCP}_{\textrm{\tiny{RT}}}))}{c_{\text{cis}}\times\sum_{t=1}^{N}d_t}.\ \ 
$$
Similarly, we compare the progression-free survival rates of radiotherapy alone against radiotherapy plus 5-FU regimens reported by two different clinical studies to estimate parameters $\theta$ and $\psi$ \cite{peters2000concurrent,eifel2004pelvic}. This yields the following two equations for estimating parameters $\theta$ and $\psi$:
$$
\theta \sum_{i=1}^{N} c^{(m)}_{i}+\psi \sum_{i=1}^{N}d^{(m)}_{i}\times c^{(m)}_{i}=-\ln\left(\frac{\ln \left(\textrm{TCP}^{(m)}_{\textrm{\tiny{CRT}}}\right)}{\ln \left(\textrm{TCP}^{(m)}_{\textrm{\tiny{RT}}}\right)} \right) \qquad m=1,2.
$$
in which the regimens used in the two studies are indexed by $m=1,2$.}

All model parameters are summarized in Table \ref{table-data}. {\color{black}Based on the survival rates reported in \cite{peters2000concurrent} and \cite{eifel2004pelvic}, we estimated an insignificant value for parameter $\psi$ (i.e., $\psi\approx 0$). Note that $\psi$ is associated with 5-FU and does not include the radio-sensitization effect of cisplatin. To study the structure of the optimal solution in the presence of radio-sensitization, we use the same radio-sensitization parameter value used for cisplatin (i.e., $\psi=\kappa$).} To fully explore how the structure of the optimal solution relates to the optimal conventional fractionation regimes in different scenarios, we consider a wide range of values for the parameters $\left[\alpha/\beta\right]$. There is a significant amount of debate as to whether or not the LQ model applies to large doses \cite{kirkpatrick2008linear,brenner2008linear}. In view of this, we set $d_{\max}$ to be 5 Gy.

\begin{table}[ht!]
	\begin{center}
    	\begin{tabular}{ | l | l | l | l | }
    	\hline
    	Structure & Parameters & Values  & Source \\ \hline
    	\multirow{8}{*}{Cervical Tumor} & $\alpha$ & $0.43$ Gy$^{-1}$ &  \cite{wigg2001applied}  \\ 
	    & $\left[\alpha/\beta\right]$ & \{3,12,20\} Gy &   \\ 
	    & $\kappa$ & $3.21\times 10^{-3} \textrm{ m}^2/(\textrm{mg} \times \textrm{Gy})$ & \cite{keys1999cisplatin}  \\ 
	    & $\theta$ & $7.15\times 10^{-5} \textrm{ m}^2/\textrm{mg}$ & \cite{peters2000concurrent,eifel2004pelvic}  \\ 
	    & $\psi$ & $\{0,3.21\times10^{-3}\}  \textrm{ m}^2/(\textrm{mg} \times \textrm{Gy})$ & \cite{keys1999cisplatin,peters2000concurrent,eifel2004pelvic}  \\ 
	    & $R$ & 20 Gy/day &  \cite{plataniotis2008use}  \\
	    & $e_g$ & 1.16 &  \cite{plataniotis2008use} \\
	    & $\sigma$ & 12 day$^{-1}$ &  \cite{plataniotis2008use} \\
	    & $\tau_d$ & $4.5$ days &  \cite{wang20062708}  \\
	    & $T_k$ & $21$ days &  \cite{plataniotis2008use}  \\
	    & $X_0$ & $10^{9}$ cells &  \cite{plataniotis2008use}  \\  
	    & $\iota$ & $10.48$ mmHg &  \cite{lartigau1998variations} \\
	    & $y_{\max}$ & $26$ mmHg &  \cite{lartigau1998variations}  \\\hline
	    \multirow{2}{*}{Bladder} & $\left[\alpha/\beta\right]$ & 2.00 &  \cite{wigg2001applied}  \\ 
	    & $\gamma$ & 60.48\% & \cite{portelance2001intensity}\\ \hline
	    \multirow{2}{*}{Small Intestine} & $\left[\alpha/\beta\right]$ & 8.00 &  \cite{joiner2009basic} \\ 
	    & $\gamma$ & 34.24\% & \cite{portelance2001intensity}\\ \hline
	    \multirow{2}{*}{Rectum} & $\left[\alpha/\beta\right]$ & 5.00 &  \cite{wigg2001applied}  \\ 
	    & $\gamma$ & 46.37\% & \cite{portelance2001intensity}\\ \hline
    	\end{tabular}
    	\caption {Tumor and normal tissue parameters }
    	\label{table-data}
	\end{center}
\end{table}

Distribution of the most common sites of cervical cancer metastases is outlined in Table \ref{table-met} (see \cite{carlson1967distant}). Metastases often grow at the same rate as the primary tumor \cite{friberg1997growth}; hence, it is natural to assume that $\mu_1=\dots=\mu_n=\frac{\ln 2}{\tau_d}$. Although distant metastases may appear any time after treatment, approximately 88\% of multi-organ metastases in cervical cancer occur within three years of the conclusion of treatment \cite{carlson1967distant}. Hence, we set parameter $T$ to three years. Parameter $\omega_i$ represents the reduction in metastatic cell growth rate at site $i$ in relation to the chemotherapy dose administrated and has a unit of $\textrm{mg}/\textrm{m}^2$ per day. There is a lack of reliable clinical data on the rate of chemotherapy-induced cell-kill in metastatic lesions. Therefore, to estimate $\omega_i$, we compare it to the drug's cell-kill effect in the primary tumor. More specifically, a chemotherapy dose of $c_t\;\textrm{mg}/\textrm{m}^2$ on day $t$ reduces the primary tumor cell population by a factor of $e^{-\theta c_t}$, in which $\theta$ has a $\textrm{m}^2/\textrm{mg}$ unit (see equation (\ref{xt})). It also reduces the metastatic cell population 24 hours after delivering the chemotherapy at site $i$ by a factor of $e^{- \omega_i c_t}$, and thus $\omega_i$ has a $\textrm{m}^2/\textrm{mg}$ per day unit (see equation (\ref{metgrowth})). Comparing the two cell-kill factors in the primary tumor and metastatic site $i$, we set the cell-kill parameter to $\omega_i=\zeta_i \theta$, in which $\zeta_i$ (with unit $\textrm{day}^{-1}$) is a positive constant depending on the drug's concentration and sensitivity of metastatic cells to chemotherapy at the $i$th site. The constant $\zeta_i =1$ represents a similar drug's concentration and cell-kill in both the primary tumor and the metastatic lesion at site $i$. For other parameters in our model, unless stated otherwise, we set $\xi=\frac{2}{3}$ and $\zeta_i=1$, and use the values specified in Table \ref{table: p} for $p_i\;(i=1,...,n)$.

\begin{table}[ht!] 
	\begin{center}
    	\begin{tabular}{ | l | l | l | l | l | }
    	\hline
    	Parameters & Nodes & Lung  & Bone & Abdomen \\ \hline
    	$p_i$ & 33.3\% & 40.30\% &  18.11\% & 8.29\%  \\ \hline
    	\end{tabular}
    	\caption {\label{table: p} Metastasis site probabilities based on data presented in \cite{carlson1967distant}}
    	\label{table-met}
	\end{center}
\end{table}

For numerical implementation of our DP algorithm, when there is no radio-sensitization and either of the conditions stated in Theorem \ref{theoremchemo} is satisfied (2D state space), we use a discretization step of 0.1 Gy for the radiation dose fractions. {\color{black}However, the DP algorithm is computationally more expensive when applied to the cases of radio-sensitizers (4D state space) or dynamic tumor radio-sensitivity parameters (6D state space). Hence, to ease the computational burden for those cases, we use discretization steps of $1 \textrm{ Gy}$ and $500\textrm{ mg}/\textrm{m}^2$ for radiation and chemotherapy dose fractions, respectively.}

\subsection{{\color{black}Optimal CRT regimens under static tumor radio-sensitivity parameters}}
Figure \ref{fig-optimalchemo} shows the optimal chemotherapy fractionation schedule {\color{black}assuming no radio-sensitization ($\psi=0$)} with identical chemotherapy cell-kill parameters at distant metastatic sites ($\zeta_1=\dots=\zeta_n$). In particular, if $\zeta_i\;(i=1,\ldots,n)$ is equal at all distant metastatic sites or only a single metastatic site is considered, then either of the conditions stated in Theorem \ref{theoremchemo} will always be satisfied. In that case, the optimal CRT regimen will deliver $c_{\max}\; \left(\textrm{mg}/\textrm{m}^2\right)$ at consecutive treatment sessions concentrated at either the beginning or the end of the planning horizon.

\begin{figure}[h]

\subfloat[]{\includegraphics[width = 75mm]{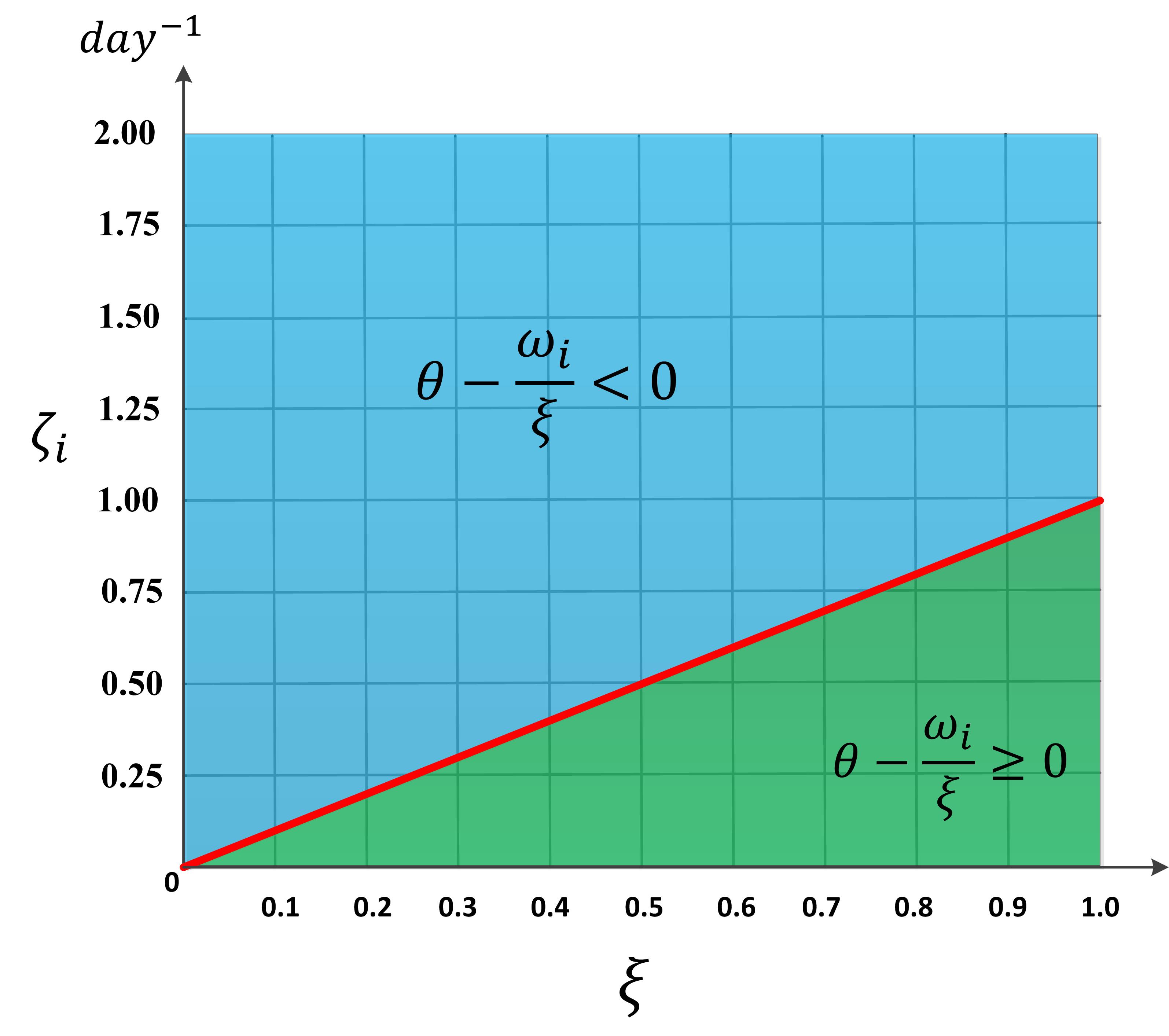}} 
\subfloat[]{\includegraphics[width = 90mm]{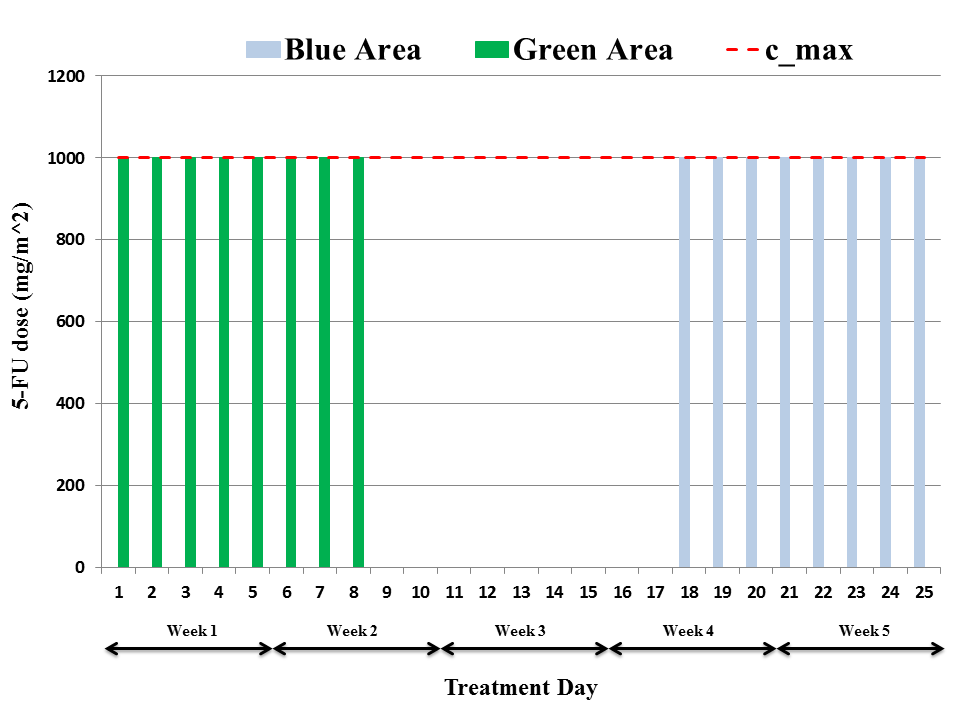}}  

\caption{(a) Schematic of optimal regimens for chemotherapy when assuming same $\omega_i$ for all distant metastasis sites {\color{black}and no radio-sensitivity effect.} Green and blue regions illustrate whether it is optimal to immediately initiate chemotherapy or postpone it until the end of the planning horizon, respectively. (b) Optimal chemotherapy fractionation schedule associated with each region displayed in Figure \ref{fig-optimalchemo}(a) for planning horizon of five weeks (25 days).}
\label{fig-optimalchemo}
\end{figure}

Figure \ref{fig-optimalradio} plots the optimal radiotherapy fractionation schedule for three different values of $\left[\alpha/\beta\right]=\{3,12,20\}$ and {\color{black}$\varrho=\{70\%,98\%\}$}. We observe that minimization of the metastatic-cell population in the absence of radio-sensitization leads to a hypo-fractionated schedule with large initial doses that taper off quickly, regardless of the values for $\left[\alpha/\beta\right]$ and $\min\{\left[\alpha/\beta\right]_j/\gamma_j\}_{j=1,\dots,M}$. If $\left[\alpha/\beta\right]$ is sufficiently large and {\color{black}$\varrho$} is close to 1, then in order to satisfy inequality \eqref{eq-bedtum}, the optimal schedule tends to be a semi-hypo-fractionation schedule with large initial doses that taper off slowly. An extensive sensitivity analysis of the radiotherapy optimal solution to model parameters ($T$, $\xi$, $N$, $T_k$) (not reported here) shows that the radiotherapy optimal schedule is robust to perturbations in these parameters and only depends on $\left[\alpha/\beta\right]$, $\left\{\left[\alpha/\beta\right]_i/\gamma_i\right\}_{i=1,\dots,M}$ and {\color{black}$\varrho$}. To study the structure of the optimal schedules in the absence of radio-sensitization and when none of the conditions stated in Theorem \ref{theoremchemo} is satisfied, we use our DP approach with a 3D state space to solve the formulation for a wide range of values of $\xi$ and $p_i$. We observe that the optimal radiotherapy schedule is determined independently of $\xi$, $\theta$, $p_i$, and $\omega_i$ and follows the same pattern as displayed in Figure \ref{fig-optimalradio} (see black arrows in Figures \ref{fig-chemosensitivity}(a) and \ref{fig-chemosensitivity}(b)).

\begin{figure}

\subfloat[{\color{black}$\varrho=70\%$}]{\includegraphics[width = 90mm]{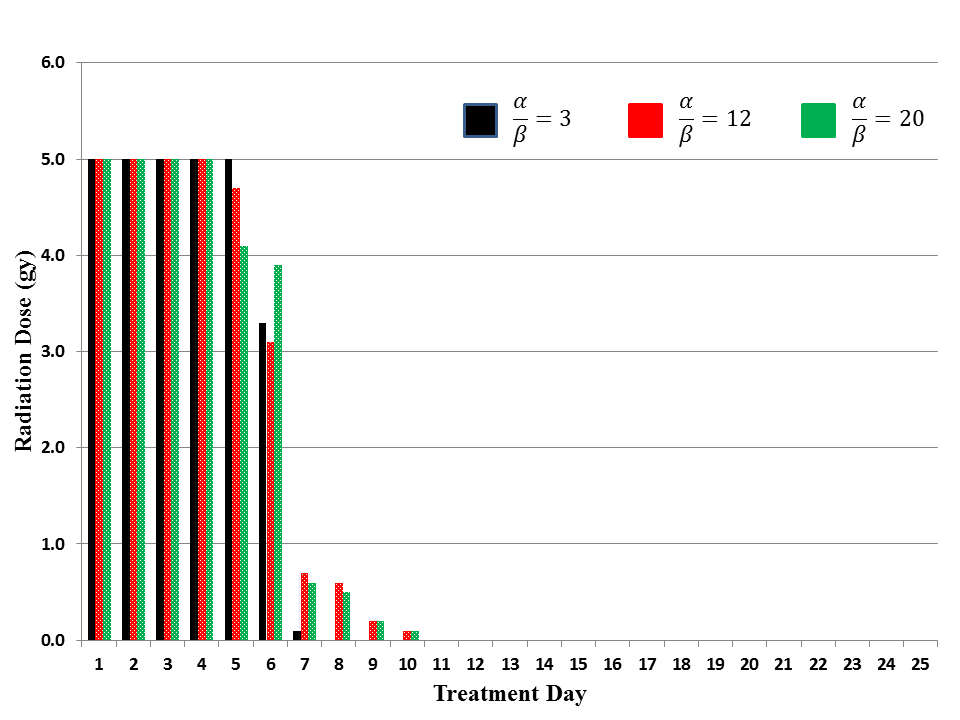}} 
\subfloat[{\color{black}$\varrho=98\%$}]{\includegraphics[width = 90mm]{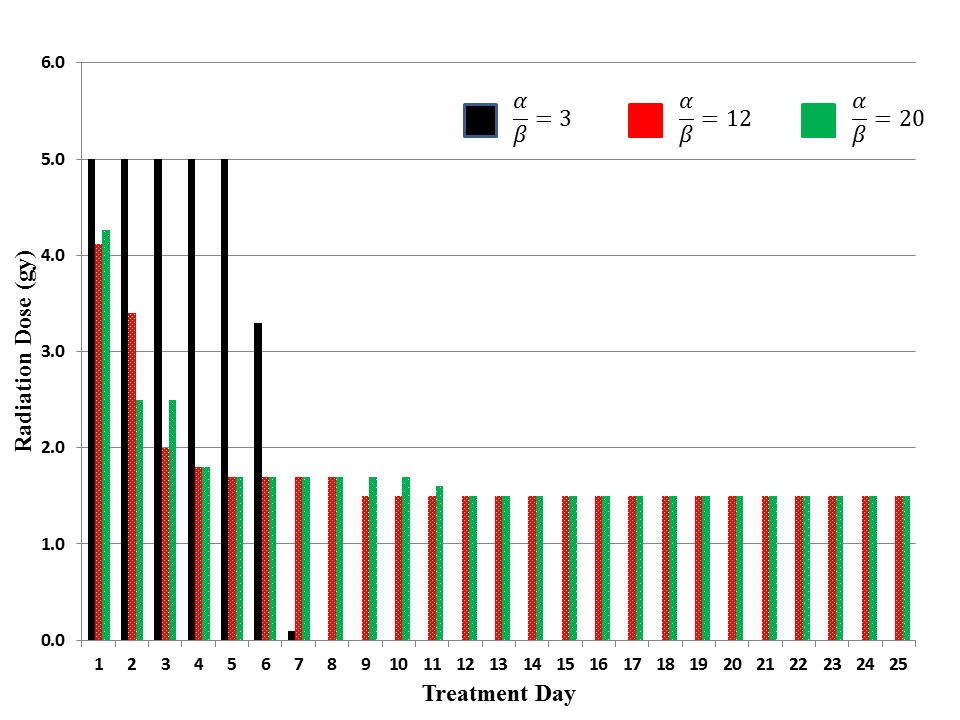}}  

\caption{Optimal radiotherapy fractionation schedule for different values of $\left[\alpha/\beta\right]$ {\color{black}in the absence of radio-sensitivity effect}. Black bars in plots (a) and (b) indicate that when $\left[\alpha/\beta \right]\le \min_j\{ \left[\alpha/\beta \right]_j/\gamma_j\}$, a hypo-fractionated schedule minimizes both tumor cell population and expected metastatic cell population. For larger values of $\left[\alpha/\beta\right]$, i.e., $\left[\alpha/\beta\right]> \min_j\{\left[\alpha/\beta\right]_j/\gamma_j\}$, where an equal-dosage routine (hyper-fractionated schedule) minimizes the number of tumor cells at the conclusion of therapy, a hypo-fractionated schedule is still the best solution to minimizing the metastatic cell population. By increasing the parameter {\color{black}$\varrho$} for tumors with large values of $\left[\alpha/\beta\right]$, in order to satisfy the BED constraint stated in \eqref{eq-bedtum}, we observe that the resulting fractionation {color{black}schedule} is a semi-hypo-fractionated structure with large initial doses that taper off slowly.}
\label{fig-optimalradio}
\end{figure}

Figure \ref{fig-chemosensitivity} illustrates optimal chemotherapy regimens for different values of $p_i$ and $\xi$ when reducing $\zeta$ of the lung to one-third for fixed values of $\left[\alpha/\beta\right]=12$ and $\varrho=70\%$. By reducing the value of $\omega$ in the lung, we ensure that a hypo-fractionated chemotherapy regimen concentrated at the beginning of therapy minimizes the metastasis population in the lung, whereas a hypo-fractionated chemotherapy regimen concentrated at the end of the planning horizon minimizes the metastasis population at three other distant metastatic sites, i.e., nodes, bone, and abdomen. Figure \ref{fig-chemosensitivity}(a) shows that {\color{black}for the case of negligible radio-sensitization,} if the probability of metastasis occurrence in the lung is small (large), then the optimal chemotherapy regimen is dominated by the schedule that minimizes expected metastatic population at the other three sites (lung), i.e., a hypo-fractionated chemotherapy regimen concentrated at the end (beginning) of the planning horizon. 

Figure \ref{fig-chemosensitivity}(b) illustrates how parameter $\xi$ changes the optimal solution in the case where conditions stated in Theorem \ref{theoremchemo} do not hold, e.g., $\xi\in(0.3,1)$. In particular, we observe that for high (low) values of $\xi$, the overall structure of an optimal solution is dominated by the schedule that minimizes the expected metastatic population in the lung (nodes, bones, and abdomen). {\color{black}This is due to the fact that as we increase parameter $\xi$, there is less incentive to schedule a chemotherapy regimen concentrating toward the end of the treatment. Hence, the optimal solution favors scheduling chemotherapy sessions at the beginning of the treatment. Figures \ref{fig-chemosensitivity}(c) and \ref{fig-chemosensitivity}(d) show that if radio-sensitization is present ($\psi>0$), then the optimal solution escalates the radiation dose in those fractions in which chemotherapy is administered. In particular, the optimal regimen shifts several radiotherapy fractions toward the end of the treatment, or schedules several chemotherapy doses at the beginning of therapy, to benefit from the radio-sensitization mechanism. Additionally, in the case of radio-sensitization, the optimal regimen is not sensitive to parameter $p_{\text{lung}}$. Based on this result, we observe that radio-sensitization favors split-course schedules in which radiotherapy and chemotherapy are administered concurrently, particularly when the metastatic dissemination has a larger rate and/or higher success probability.}

\begin{figure}
\centering
\includegraphics[width=170mm]{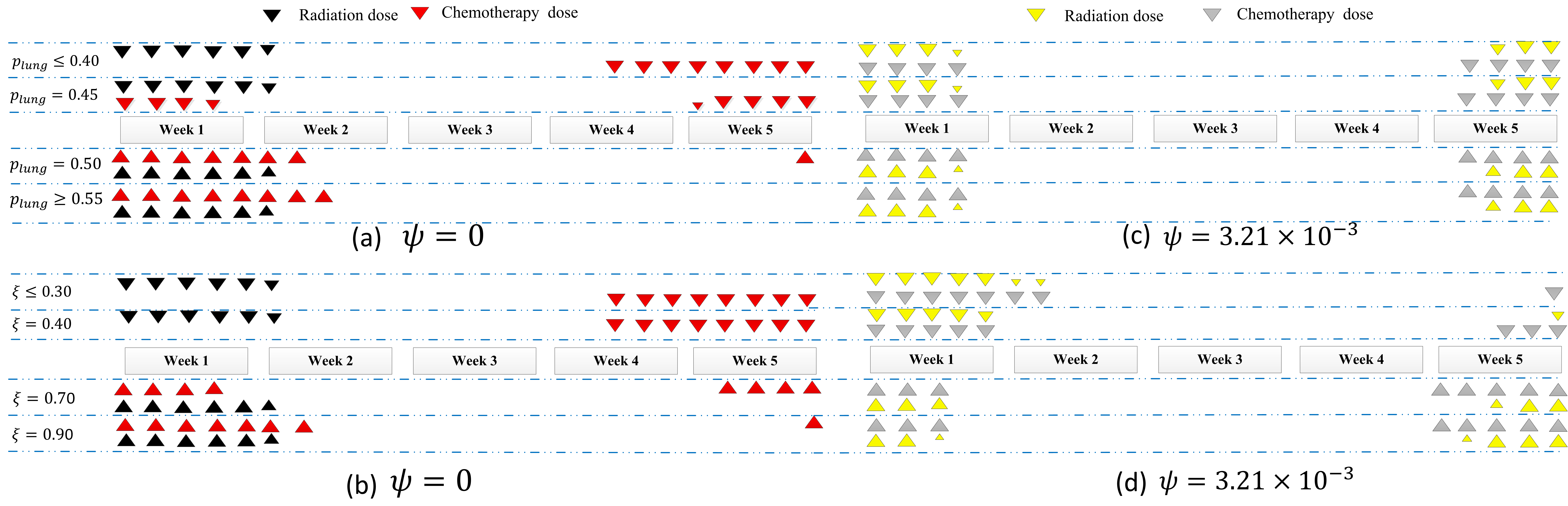}
\caption{{\color{black}Optimal chemotherapy regimen when conditions stated in Theorem \ref{theoremchemo} do not hold. For these examples, we assume that $\omega_{\text{nodes}}=\omega_{\text{bone}}=\omega_{\text{abdomen}}=\theta$, $\omega_{\text{lung}}=0.3\times\theta$, $\left[\alpha/\beta\right]=12$, $\varrho=70\%$ and $\psi=\{0,3.21\times 10^{-3}\}$. (a) and (c) Optimal chemotherapy regimen for different values of $p_{\text{lung}}$. (b) and (d) Optimal chemotherapy regimen for different values of $\xi$. Red and gray arrows represent chemotherapy doses, whereas black and yellow arrows represent radiotherapy doses. The arrow position represents the time of dose during the Monday-to-Friday treatment window. The size of the arrow correlates with the size of the 5-FU or radiation dose, given that the maximum daily dose cannot exceed $1,000\;\textrm{m}^2/\textrm{mg}$ and $5$ Gy for chemotherapy and radiation, respectively.}}
\label{fig-chemosensitivity}
\end{figure}

We consider the relative effectiveness of an optimized schedule versus a standard schedule (delivering 45 Gy of radiotherapy with 1.8 Gy fractions per day in five weeks and $8,000\; \textrm{m}^2/\textrm{mg}$ of 5-FU with  $1,000\; \textrm{m}^2/\textrm{mg}$ per day delivered on Tuesday, Wednesday, Thursday, and Friday of the first and fourth weeks \cite{peters2000concurrent}). In particular, we denote the approximate metastatic population at $n$ sites under the optimized schedule when setting {\color{black}$\varrho=x\%$} in \eqref{eq-bedtum} by $R_{\text{opt}}^x$ and the metastatic population under a standard uniform fractionation by $R_{\text{std}}$. Then we denote the BED delivered to the primary tumor under optimized and standard schedules by $\textrm{BED}_{\textrm{opt}}^x$ and $\textrm{BED}_{\textrm{std}}$, respectively. The ratios $(R_{\textrm{std}}-R_{\textrm{opt}}^x)/R_{\textrm{std}}$ and $(\textrm{BED}_{\textrm{std}}-\textrm{BED}_{\textrm{opt}}^x)/\textrm{BED}_{\textrm{std}}$ will give us a measure of the predicted relative reduction in metastasis population and tumor BED, respectively, associated with using the optimized schedule with {\color{black}$\varrho=x\%$} instead of the standard schedule. The results illustrated in Figure \ref{fig-improv} show that the choice of {\color{black}$\varrho < 100\%$} can lead to a significant reduction in the metastasis population for all values of $\left[\alpha/\beta\right]$ and $\xi$. Note that for $\left[\alpha/\beta\right]\ge (\min_i{\left[\alpha/\beta\right]_i/\gamma_i}=3.3)$, the standard regimen is indeed the optimal schedule for the conventional fractionation problem that maximizes TCP \cite{badri2015optimal,badriglioma}, whereas a hypo-fractionated radiotherapy schedule minimizes the metastatic population size. Hence, Figure \ref{fig-improv} illustrates the trade-off between the two conflicting objectives: (i) minimizing metastatic population size and (ii) maximizing tumor TCP for different $\left[\alpha/\beta\right]$ and $\xi$ values. One can observe a diminishing return in the reduction of the metastatic population. In particular, the fractionation solution obtained by setting {\color{black}$1-\varrho=4\%$} seems to yield an interesting trade-off between the two objectives beyond which allowing for larger tumor BED reductions, i.e., using smaller {\color{black}$\varrho=x\%$}, does not lead to any significant gain in metastatic-population reduction (with the exception of very small values of $\xi$). Last, for the tumor $\left[\alpha/\beta\right] < 3.3$, the two objectives (i) and (ii) are not of a conflicting nature since a hypo-fractionated regimen is desired by both objectives. 

\begin{figure}
\centering
\subfloat[{Varying $\left[\alpha/\beta\right]$, $\xi=\frac{2}{3}$}]{\includegraphics[width = 85mm]{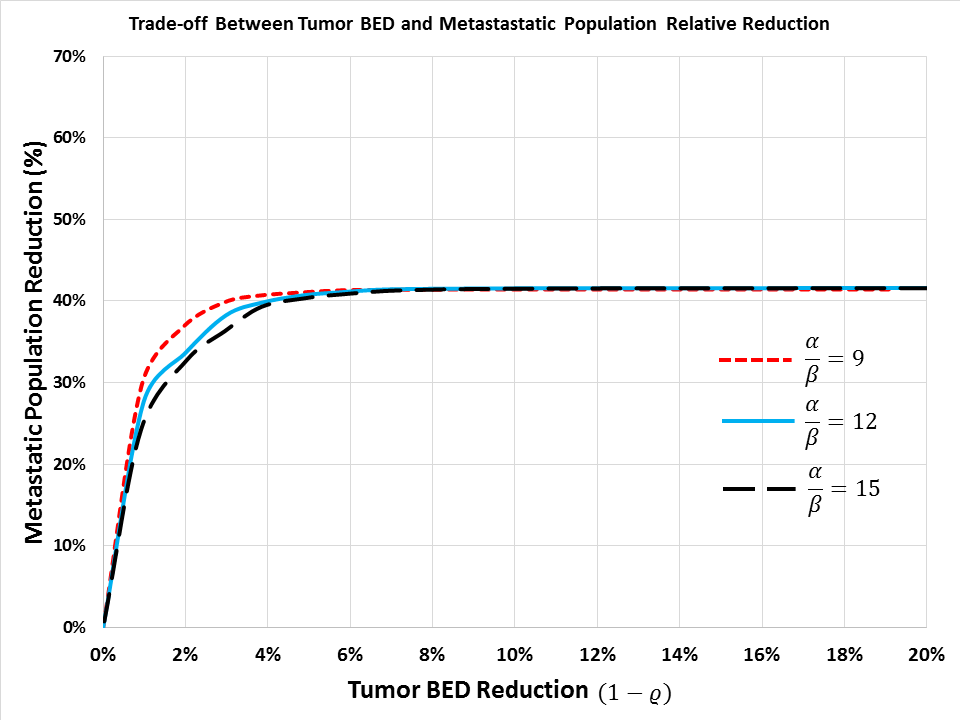}} 
\subfloat[{Varying $\xi$, $\left[\alpha/\beta\right]=12$}]{\includegraphics[width = 85mm]{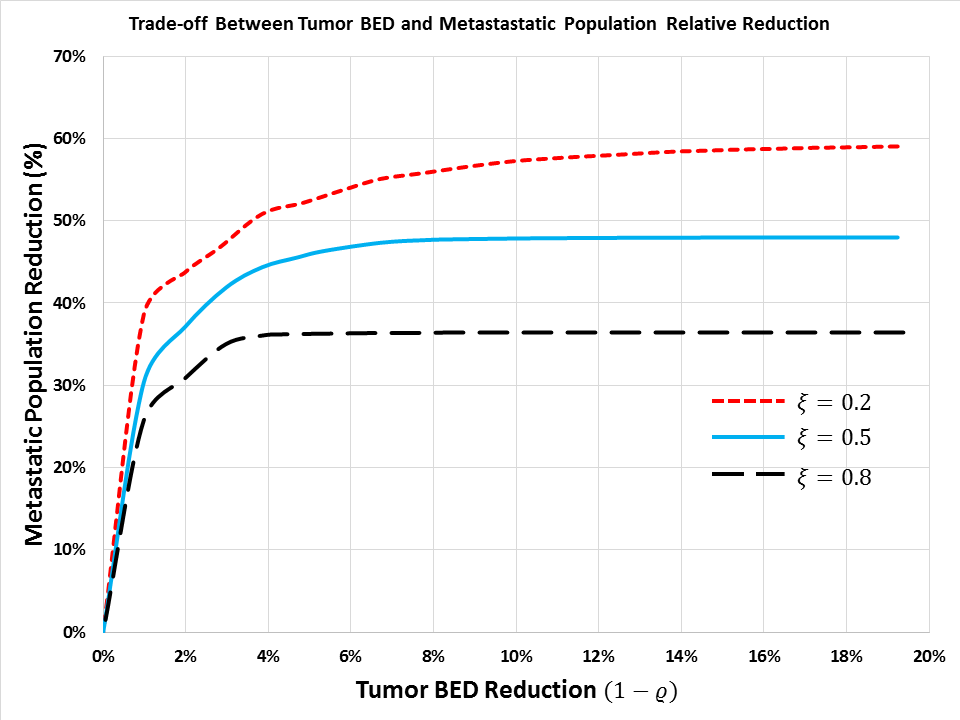}}  
\caption{Trade-off between reduction in metastatic population size and tumor BED relative to optimal conventional protocols {\color{black}assuming $\psi=0$}. (a) Maximum metastatic population reduction for all values of $\left[\alpha/\beta\right]$ examined in this plot is $42\%$ (obtained at {\color{black}$\varrho=0\%$}). (b) Maximum metastatic population reduction (obtained at {\color{black}$\varrho=0\%$}) for $\xi=0.2$, $0.5$, and $0.8$ examined in this plot is $59\%$, $48\%$, and $36\%$, respectively.}
\label{fig-improv}
\end{figure}
{\color{black}
\subsection{Optimal CRT regimens under dynamic tumor radio-sensitivity parameters}
Figure \ref{fig-dynamic} illustrates optimal CRT regimens for the case in which the tumor radio-sensitivity parameters $\alpha$ and $\beta$ change throughout the treatment course due to tumor re-oxygenation, as shown by equation \eqref{eq-OER}. Under dynamic radio-sensitivity parameters, it is optimal to immediately start chemotherapy treatment. Additionally, a careful comparison of Figures \ref{fig-chemosensitivity} and \ref{fig-dynamic} reveals that in contrast to static radio-sensitivity parameters, the structure of the optimal chemotherapy schedule is robust to changes in model parameters $p_{lung}$ and $\xi$. With respect to the radiotherapy schedule, it is optimal to deliver a few large fractions at the beginning of the treatment course followed by several smaller radiation doses in the second and third weeks of therapy. We believe this can be explained by the increasing order of $\alpha_t$ and $\beta_t$ values over the treatment course. In fact, there is an exponential increase in $\alpha_t$ and $\beta_t$ at the beginning of the treatment course, reaching maximum possible values after delivering a few large fractions of chemotherapy and radiation (within three or four sessions). Hence, in the optimal schedule, a few large radiation and chemotherapy fractions are administered at the beginning of the therapy to reduce the tumor size significantly. This leads to a well-oxygenated tumor at subsequent radiation fractions, rendering the remaining tumor population more susceptible to radiation. Thus, to benefit from this phenomenon, the optimal regimen administers a set of small radiation fractions during later sessions.

\begin{figure}
\centering
\includegraphics[width=120mm]{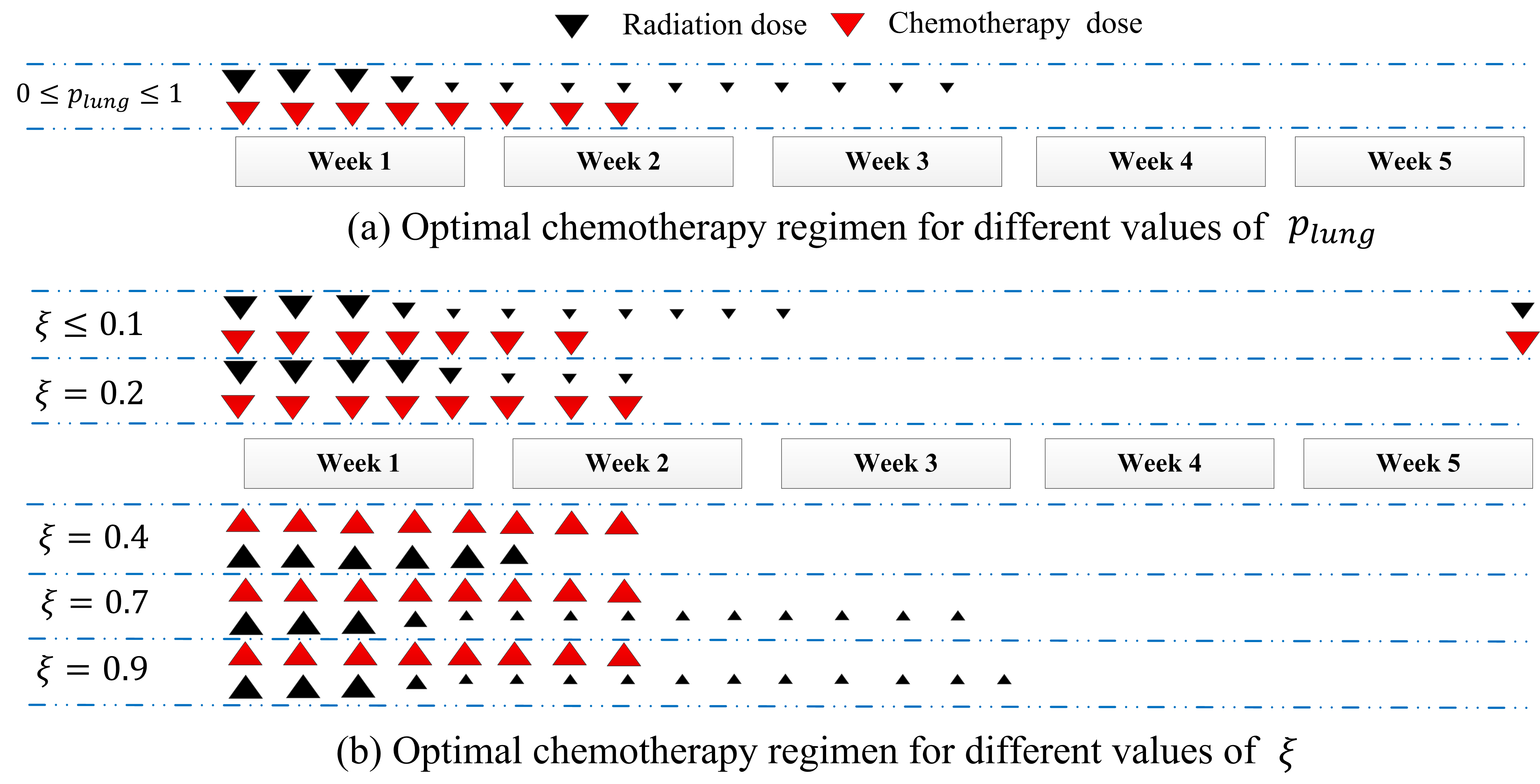}
\caption{ {\color{black}Optimal CRT regimens with dynamic radio-sensitivity parameters, assuming that $\omega_{\text{nodes}}=\omega_{\text{bone}}=\omega_{\text{abdomen}}=\theta$, $\omega_{\text{lung}}=0.3\times\theta$, $\alpha_{\max}=0.45$, $\left[\alpha_{\max}/\beta_{\max}\right]=12$, $\varrho=70\%$, $\psi=3.21\times 10^{-3}$, $\rho=10^8$ \cite{wein2000dynamic}, $OER_{\alpha}=2.5$, $OER_{\beta}=3$ and $K=3.28$ \cite{wouters1997cells}.}}
\label{fig-dynamic}
\end{figure}
}
\section{Conclusion} \label{sec: conclusion}
In previous work that considers optimal fractionation of chemotherapy and radiotherapy, the goal has been to maximize the probability of controlling the primary tumor, i.e., local control. Based on the observation that the majority of cancer fatalities are due to metastatic {\color{black}disease \cite{gupta2006cancer}}, we consider an alternative objective: design CRT fractionation schedules that minimize the metastatic population over a sufficiently long window of time. 
The current work is a significant extension of our previous work  \cite{badri2015minimizing}, which incorporates multi-site metastatic disease as well as the effect of chemotherapy through additivity, radio-sensitization, and spatial-cooperation mechanisms. {\color{black}We also account for tumor re-oxygenation during the course of CRT treatments.} We model the metastasis population as a multi-type non-homogeneous branching process, where each successful metastatic cell is able to colonize a new tumor at one of several potential distant sites with a known probability. We are able to derive closed-form solutions to optimal chemotherapy regimens and prove an interesting structure of optimal radiotherapy schedules under easily verified conditions. We numerically solve this optimization problem using a DP approach.

The resulting optimal schedules have an interesting structure that is quite different from that observed in the traditional optimal fractionation problem, where the goal is to minimize the local tumor population at the conclusion of therapy. In the traditional optimal fractionation problem, it was shown that if the tumor $\left[\alpha/\beta\right]$ ratio is smaller (bigger) than the effective $\left[\alpha/\beta\right]$ ratio of all OAR, then a hypo-fractionated (hyper-fractionated) schedule is optimal \cite{mizuta2012mathematical}. {\color{black}For the case of CRT with only additivity and spatial cooperation mechanisms,} we proved that the optimal radiotherapy schedule has a non-increasing structure and that it is optimal to immediately start radiotherapy treatment. Also, we numerically observed that, a hypo-fractionated radiotherapy regimen with large initial doses that taper off quickly, minimizes the metastatic population, regardless of tumor and normal-tissue $\left[\alpha/\beta\right]$ ratios. This result is partially consistent with our previous work, where the benefit of hypo-fractionation was observed for low values of the tumor $\left[\alpha/\beta\right]$ ratio and high values of the tumor $\left[\alpha/\beta\right]$ ratio, if the length of time for which we evaluated metastatic risk is short \cite{badri2015minimizing}. In the current work, we observed that if we consider metastases as actively growing tumor cell colonies and model them as a non-homogeneous branching process, then regardless of evaluation period and magnitude of tumor $\left[\alpha/\beta\right]$ ratio, our results show that a hypo-fractionated schedule is nearly always optimal. We proved that it is always optimal to deliver the maximum drug concentration allowed during the course of therapy, $C_{\max}$. 

In the setting of no radio-sensitization ($\psi=0$) and static radio-sensitivity parameters we established results on the structure of the optimal chemotherapy schedule. When the value $\theta\xi$, where $\theta$ is the chemotherapy-induced cell-kill at the primary site and $\xi$ is the fractal dimension of the primary tumor cells capable of metastasizing, is larger than the chemotherapy-induced cell-kill at all distant sites ($\max_{i=1,\dots,n}\{\omega_i\}$), the optimal chemotherapy fractionation regimen suggests delivering $C_{\max}$ in consecutive days starting from the first day of the planning horizon and administering the maximum tolerable daily dose at each fraction (concurrent regimen). However, under condition $\theta\xi<\min_{i=1,\dots,n}\{\omega_i\}$, it is optimal to postpone delivering chemotherapy (or part of the chemotherapy along with few radiation doses in the presence of radio-sensitization) until the end of treatment (adjuvant regimen). We observed that when these two conditions are not satisfied, the optimal regimen is delivering a portion of the total chemotherapeutic agents at the beginning of the planning horizon and administrating the remaining amount at the end of the therapy. This portion is a function of $\theta$, $\xi$, the probability that a successful metastatic cell colonizes at a specific distant site and $\omega_i$ for each site.

{\color{black}Our sensitivity analysis for cervical tumors reveals that chemotherapeutic agents with no radio-sensitization property do not change the optimal radiation fractionation regimens, and vice-versa. However in the case of active radio-sensitization effect, optimal regimens use an escalated radiation dose on treatment sessions with chemotherapy administrations to benefit from the radio-sensitization mechanism. This can be compared to an earlier work by Salari et al.\ where it was shown that radio-sensitizers may alter the optimal radiation fractionation regimens in a similar fashion \cite{salari2013mathematical}. {\color{black} The benefits of accelerated hypo-fractionation schedules was established numerically, however radio-sensitizers may postpone delivering few radiotherapy sessions until the end of therapy.} Moreover, we consider an extended version of the model with dynamic tumor radio-sensitivity parameters due to tumor re-oxygenation throughout the treatment course.} {\color{black} This leads to optimal CRT regimens in which the maximum chemotherapy tolerable dose are administered consecutively at the beginning of treatment; and few large radiation fractions are administered during the first week of treatment, followed by small doses of radiation in the second and third weeks of therapy.}

Interestingly, previous clinical trials (see e.g., \cite{green2001survival} for a review) show the benefit of CRT on overall and progression-free survival, and local and distant control in patients with cervical cancer. In particular, in clinical trials, it was observed that there is a significant reduction in the rate of distant metastases in patients diagnosed with cervical cancer treated with both platinum and non-platinum chemotherapy. This reduction was achieved with a short course of chemotherapy combined with local treatment. However, there is presently no evidence that neoadjuvant CRT reduces the incidence of distant metastases \cite{green2001survival}, which is consistent with our result discussed in Remark 1 of Section 3.

Our goal here is not to recommend alternative clinical practice but to assist clinicians with hypothesis generation to design novel CRT fractionation schemes that can be tested in clinical trials. Our numerical results suggest that using optimal schedules instead of standard regimens has the potential of reducing the metastatic population by more than 40\%, and even for some cases where the tumor $\left[\alpha/\beta\right]$ ratio is small, we can improve the tumor BED as well. 

This paper is an initial step toward developing a new framework to incorporate the risk of metastatic disease into CRT fractionation decisions. What follows is a discussion of some limitations in our framework, which are partly imposed by the lack of clinically established models of the underlying biological processes. First, it is assumed that the chemotherapeutic drugs administered at previous treatment fractions do not carry over to the current fraction. However, this may not necessarily be a valid assumption. Therefore, an important extension of this work will be to incorporate the effects of previously administered drugs to subsequent fractions. A possible solution to this could be modeling the effective dose at each fraction as a weighted moving average of the current and previously administrated doses. {\color{black}Second, the process of metastatic formation and growth at distant sites was modeled using a non-homogeneous Poisson process and a multi-type branching process, respectively, which were adopted from previous studies. Those models assume that the metastatic cells act independently from each other for analytical tractability. However, metastatic lesions may compete over limited oxygen and nutrient supplies required for growth, rendering the independence assumption unfavorable. Last, a simple linear model was used to describe the relationship between the tumor radius and its average oxygen partial pressure. In particular, the model assumes a dynamic and yet uniform oxygen partial pressure that changes linearly as the tumor radius shrinks. The assumptions of the uniformity of the oxygen pressure and its linear relationship to tumor radius may be overly simplistic; however, this linear model is mainly intended to mimic the increasing behavior of tumor oxygen pressure during therapy, as suggested by some clinical studies \cite{lartigau1998variations}.}

This work considers the problem of finding radiation and chemotherapy schedules that optimally minimize the total metastatic populations. While the parameter values for the present work are focused on cervical cancer, our work is applicable to a wider range of cancers that are treated using CRT. In particular, we are very eager to investigate the application of our model to additional cancers in order to find the optimal CRT regimens and study how they may vary across different cancers. Furthermore, our model is based on the assumption that the total population of the primary tumor can be accurately approximated with a deterministic function. However, the nature of tumor response to chemotherapeutic agents and ionizing radiation is stochastic and varies across different patients. Therefore, a potentially interesting extension of this work could be modeling the primary tumor response to CRT as a stochastic process.

{\color{black}While the decision problem modeled and solved in this work focuses on cancer treatment, it can be classified as belonging to a broader category of decision problems related to the optimal control of spatiotemporal spread of threats. Other forms of threats include infectious diseases, invasive spices, and Internet security threats. The common goal is to find the optimal combination and timing of intervention strategies that minimize the threat dispersal, while ensuring that budget (resource) limitations are met. The intervention strategies may involve, among others, slowing down the dispersion rate, directly targeting localized outbreaks, and surveillance. The modeling and solution approach developed in this study can be, in principle, extended and applied to other decision problems in this category.}

\section{Acknowledgments}
We would like to thank Dr. Emil Lou and Dr. Jianling Yuan for their valuable feedback on this manuscript. In particular, they gave us very helpful comments on the current clinical practice in the treatment of cervical cancer.
\clearpage
 \bibliographystyle{plain}
 \bibliography{Refs}

\clearpage
\section*{Appendices}
\renewcommand{\thesubsection}{\Alph{subsection}}

\subsection{Proof of Lemma \ref{lemmasum}} \label{lemmasumProof}
\begin{proof}
We use contradiction to prove our result. Let the optimal dose vector be $\vec{c'}=\{c_1',\dots,c_N'\}$ such that $\sum_{t=1}^{N}c_t'<C_{\max}$. Then, we can always construct a feasible solution $\vec{c''}$ based on $\vec{c'}$ that contradicts the optimality of $\vec{c'}$. For example, first we set $\vec{c''}=\vec{c'}$, and then for an arbitrary integer $i$, $1\le i\le N$, such that $c_i'<c_{\max}$ (note that this integer must exist since we have $C_{\max}<Nc_{\max}$), we set $c_i''=\min\{c_{\max},c_i'+C_{\max}-\sum_{t=1}^{N}c_t'\}$ and observe that {\color{black}
$$g\left( \sum_{t=1}^{N}d_t,\sum_{t=1}^{N}d_t^2,\sum_{t=1}^{N}c_t'',\sum_{t=1}^{N}d_tc_t''\right)\le g\left( \sum_{t=1}^{N}d_t,\sum_{t=1}^{N}d_t^2,\sum_{t=1}^{N}c_t',\sum_{t=1}^{N}d_tc_t'\right)\text{ and }f(\vec{d},\vec{c''})\le f(\vec{d},\vec{c'})$$}
since $\alpha$, $\beta$, $\theta$, and {\color{black}$\psi$} are always positive values.
\end{proof}
{\color{black}
\subsection{DP algorithm for the case of static radio-sensitivity parameters} \label{algorithm}

\begin{algorithm}
\caption{DP algorithm to solve problem \eqref{eq-obj1}--\eqref{eq-checon1} with static $\alpha$ and $\beta$}
\begin{algorithmic}[1]\label{alg1}
  \scriptsize
  \STATE {\color{black}Create a data table with headers: $T$=\{$day, U,V,S,W,d,c,obj,id,track$\}}
  \STATE {\color{black}$counter\leftarrow 1$}
  \STATE {\color{black}For $i\in \{0:d_{\min}:d_{\max}\}$ do:}
  \STATE {\color{black}$\mbox{ }$ For $j\in \{0:c_{\min}:c_{\max}\}$ do:}
  \STATE {\color{black}$\mbox{ }\quad$ Create a temporary data table with headers: $T_{\text{temp}}=\{day, U,V,S,W,d,c,obj,id,track\}$}
  \STATE {\color{black}$\mbox{ }\quad$ $T_{\text{temp}}.day\leftarrow 1,T_{\text{temp}}.U\leftarrow i, T_{\text{temp}}.V\leftarrow i^2, T.S\leftarrow j, T_{\text{temp}}.W\leftarrow i\times j, T_{\text{temp}}.d\leftarrow i, T_{\text{temp}}.c\leftarrow j$, $T_{\text{temp}}.id\leftarrow counter$, $T_{\text{temp}}.track\leftarrow 0$, and
  $$T_{\text{temp}}.obj\leftarrow X_0^\xi \sum_{i=1}^{n}p_ie^{\mu_iT-\omega_iC_{\max}}\left(1+e^{-\mu_i}+ e^{-\xi\left( \alpha i+\beta i^2+(\theta-\omega_i/\xi) j+\psi ij-\frac{\ln 2}{\tau_d}(2-T_k)^+\right)-2\mu_i}\right)  $$}
  \STATE {\color{black}$\mbox{ }\quad$ $counter\leftarrow counter+1$}
  \STATE {\color{black}$\mbox{ }\quad$ $concatenate(T,T_{\text{temp}})$}
  \STATE {\color{black}For $t\in \{2:1:N\}$ do:}
  \STATE {\color{black}$\mbox{ }$ Set $T'=\{\text{all rows in table } T|T.day=t-1\}$}
  \STATE {\color{black}$\mbox{ }$ For $k\in\{1:1:\text{height}(T')\}$ do:}
  \STATE {\color{black}$\mbox{ }\quad$ For $i\in \{0:d_{\min}:d_{\max}\}$ do:}
  \STATE {\color{black}$\mbox{ }\quad\quad$ For $j\in \{0:c_{\min}:c_{\max}\}$ do:}
  \STATE {\color{black}$\mbox{ }\quad\quad\quad$ Create a temporary data table with headers: $T_{\text{temp}}=\{day, U,V,S,W,d,c,obj,id,track\}$}
  \STATE {\color{black}$\mbox{ }\quad\quad\quad$ Set 
  $$
  T_{\text{temp}}.day\leftarrow t, T_{\text{temp}}.U\leftarrow T'.U(k)+i, T_{\text{temp}}.V\leftarrow T'.V(k)+i^2, T_{\text{temp}}.S\leftarrow T'.S(k)+j, T_{\text{temp}}.W\leftarrow T'.W(k)+i\times j, 
  $$ 
  $$
 T_{\text{temp}}.d\leftarrow i, T_{\text{temp}}.c\leftarrow j,
  \Psi_{t,i}=\mu_i\cdot(T-t)-\omega_iC_{\max}-\xi\times( \alpha T_{\text{temp}}.{U}+\beta T_{\text{temp}}.{V}+(\theta-\omega_i/\xi) T_{\text{temp}}.S+\psi T_{\text{temp}}.W-\frac{\ln 2}{\tau_d}(t-T_k)^+)
  $$
  $$ 
  T_{\text{temp}}.obj\leftarrow T'.obj(k)+\begin{cases} X_0^\xi \sum_{i=1}^{n}p_ie^{\Psi_{t,i}} & t\le N-1 \\
  X_0^\xi \sum_{i=1}^{n}p_ie^{\Psi_{t,i}}+g(T_{\text{temp}}.\hat{U},T_{\text{temp}}.\hat{V},T_{\text{temp}}.S,T_{\text{temp}}.W) & t=N \end{cases}
  $$ }
  \STATE {\color{black}$\mbox{ }\quad\quad\quad$ If (all feasibility constraints satisfied):}
  \STATE {\color{black}$\mbox{ }\quad\quad\quad\quad$ Set {$index=\{T.id|T.day=t ,T.U=T_{\text{temp}}.U,T.V=T_{\text{temp}}.V,T.S=T_{\text{temp}}.S,T.W=T_{\text{temp}}.W\}$}}
  \STATE {\color{black}$\mbox{ }\quad\quad\quad\quad$ If ($index==NULL$):} 
  \STATE {\color{black}$\mbox{ }\quad\quad\quad\quad\quad$ $T_{\text{temp}}.id\leftarrow counter $}
  \STATE {\color{black}$\mbox{ }\quad\quad\quad\quad\quad$ $T_{\text{temp}}.track\leftarrow T'.id(k) $}
  \STATE {\color{black}$\mbox{ }\quad\quad\quad\quad\quad$ $concatenate(T,T_{\text{temp}})$}
  \STATE {\color{black}$\mbox{ }\quad\quad\quad\quad\quad$ $counter\leftarrow counter+1$}
  \STATE {\color{black}$\mbox{ }\quad\quad\quad\quad$ Elseif ($ T_{\text{temp}}.obj<T.obj(T.id==index)$):} 
  \STATE {\color{black}$\mbox{ }\quad\quad\quad\quad\quad\quad$ $T.d(T.id==index)\leftarrow T_{\text{temp}}.d$}
  \STATE {\color{black}$\mbox{ }\quad\quad\quad\quad\quad\quad$ $T.c(T.id==index)\leftarrow T_{\text{temp}}.c$}
  \STATE {\color{black}$\mbox{ }\quad\quad\quad\quad\quad\quad$ $T.obj(T.id==index)\leftarrow T_{\text{temp}}.obj$}
  \STATE {\color{black}$\mbox{ }\quad\quad\quad\quad\quad\quad$ $T.track(T.id==index)\leftarrow T'.id(k)$}
  \STATE {\color{black}$id^*\leftarrow\left\lbrace  T.id|T.obj=\min\{T.obj|T.day=N\} \right\rbrace $}
  \STATE {\color{black}Return $d^*_N\leftarrow T.d(T.id==id^*)$, $c^*_N\leftarrow T.c(T.id==id^*)$}
  \STATE {\color{black}For $t\in \{N-1:-1:1\}$ do:}
  \STATE {\color{black}$\mbox{ }$ $id^*\leftarrow T.track(T.id==id^*)$}
  \STATE {\color{black}$\mbox{ }$ Return $d^*_{t}\leftarrow T.d(T.id==id^*)$, $c^*_t\leftarrow T.c(T.id==id^*)$}
\end{algorithmic}
\end{algorithm}

The DP algorithm employs a dynamic table to only store unique and feasible states at each stage. To this end, all discretized states are enumerated and only those that have unique attribute combinations (i.e., $T.U, T.V, T.S, T.W, T.\hat{U},$ and $T.\hat{V}$) are stored.  The pseudo code is provided in Algorithm 1. $d_{min}$ and $c_{min}$ represent the discretization step length for radiation and chemotherapy dose fractions, respectively.}

\subsection{Proof of Lemma \ref{lemmaordering}}\label{lemmaorderingProof}
\begin{proof}
We use contradiction to prove our result. Assume that there exists an optimal dose vector $\vec{d^*}=\{d_1^*,\dots,d_i^*,\dots,d_j^*,\dots,d_{N}^*\}$ such that for some $j>i$, we have $d_j^*>d_i^*$. Now we construct a new feasible solution $\vec{d'}=\{d_1',\dots,d_{N}'\}$, where $d_i'=d_j^*$, $d_j'=d_i^*$ and $d_k'=d_k^*$ for $k\not=i,j$ with the identical drug vector $\vec{c}$. {\color{black}Since we have $\psi=0$,} the feasibility constraints and the value of function $g$ in \eqref{eq-obj1} are order independent, i.e., if we change the order of the dose vector, then the resulting dose vector is still feasible with the same value for function $g$; hence, $\vec{d'}$ is a feasible solution to \eqref{eq-obj1}--\eqref{eq-checon1} where {\color{black}
$$
g\left(\sum_{t=1}^{N}d_t',\sum_{t=1}^{N}(d_t')^2,\sum_{t=1}^{N}c_t,\sum_{t=1}^{N}d_t'c_t \right)=g\left(\sum_{t=1}^{N}d_t^*,\sum_{t=1}^{N}(d_t^*)^2,\sum_{t=1}^{N}c_t,\sum_{t=1}^{N}d_t^*c_t \right).
$$ }
Thus, to prove the result, it is sufficient to show that $f(\vec{d'},\vec{c})\le f(\vec{d^*},\vec{c})$, i.e., \small
$$ \sum_{t=0}^{N+1}e^{-\xi\left( \sum_{j=1}^{t-1}(\alpha d_j'+\beta (d_j')^2+(\theta-\omega_i/\xi) c_j)-\frac{\ln 2}{\tau_d}(t-T_k)^+\right)-\mu_it-\omega_i C_{\max}}\le  \sum_{t=0}^{N+1}e^{-\xi\left( \sum_{j=1}^{t-1}(\alpha d_j^*+\beta (d_j^*)^2+(\theta-\omega_i/\xi) c_j)-\frac{\ln 2}{\tau_d}(t-T_k)^+\right)-\mu_it-\omega_i C_{\max}}$$\normalsize
for all $i=1,\dots,n$. The above inequality is implied by
$$\alpha\sum_{j=1}^{t-1}d_j'+\beta\sum_{j=1}^{t-1}(d_j')^2\ge\alpha\sum_{j=1}^{t-1}d_j^*+\beta\sum_{j=1}^{t-1}(d_j^*)^2 \text{ for all } t=1,\dots,N+1 \text{ and } i=1,\dots,n$$
We can use the same approach to prove the same result for $\vec{c^*}$ (see proof of Theorem \ref{theoremchemo} for more details).
\end{proof}

\subsection{Proof of Theorem \ref{theoremchemo}}\label{theoremchemoProof}
\begin{proof}
We first prove our results when $\theta\xi\ge\max_{i=1,\dots,n} \omega_i$. This condition implies that $\theta-\omega_i/\xi\ge0$ for all $i=1,\dots,n$. First, we use contradiction to show that, in optimality, we must have $c_1^*=\dots=c_k^*=c_{\max}$ for $k=\lfloor C_{\max}/c_{\max}\rfloor$. Assume that there exists an optimal dose vector $\vec{c^*}=\{c_1^*,\dots,c_i^*,\dots,c_j^*,\dots,c_{N}^*\}$ such that for some $j>i$, we have $c_i^*<c_{\max}$ and $c_j^*>0$. Now we construct a new feasible solution $\vec{c'}=\{c_1',\dots,c_{N}'\}$, where $c_i'=\min\{c_{\max},c_i^*+c_j^*\}$, $c_j'=c_j^*-(c_i'-c_i^*)$ and $c_k'=c_k^*$ for $ k\not=i,j$ with the identical dose vector $\vec{d}$. First, note that the drug vector $\vec{c'}$ is a feasible solution, since $\sum_{t=1}^{N}c_t'=\sum_{t=1}^{N}c_t^*$ and $c_t'\le c_{\max}$ for $t=1,\dots,N$. Second, the value of function $g$ in \eqref{eq-obj1} has the same value for both $\vec{c^*}$ and $\vec{c'}$, since $\sum_{t=1}^{N}c_t^*=\sum_{t=1}^{N}c_t'$ {\color{black}and $\psi=0$, i.e.,
$$
g\left(\sum_{t=1}^{N}d_t,\sum_{t=1}^{N}d_t^2,\sum_{t=1}^{N}c_t^*,\sum_{t=1}^{N}d_tc_t^* \right)=g\left(\sum_{t=1}^{N}d_t,\sum_{t=1}^{N}d_t^2,\sum_{t=1}^{N}c_t',\sum_{t=1}^{N}d_tc_t' \right).
$$  }
Thus, to prove that $\vec{c'}$ results in a smaller objective function, it is sufficient to show that $f(\vec{d},\vec{c'})\le f(\vec{d},\vec{c^*})$, i.e.,
\begin{align*} 
&\sum_{t=0}^{N+1}e^{-\xi\left( \sum_{j=1}^{t-1}(\alpha d_j+\beta d_j^2+(\theta-\omega_i/\xi) c_j')-\frac{\ln 2}{\tau_d}(t-T_k)^+\right)-\mu_it-\omega_i C_{\max}}\le \\ &  \sum_{t=0}^{N+1}e^{-\xi\left( \sum_{j=1}^{t-1}(\alpha d_j+\beta d_j^2+(\theta-\omega_i/\xi) c_j^*)-\frac{\ln 2}{\tau_d}(t-T_k)^+\right)-\mu_it-\omega_i C_{\max}}
\end{align*}
for all $i=1,\dots,n$. The above inequality can be shown if and only if we show the following inequality for $ t=0,\dots,N+1$
\begin{align*}
&-\xi\left( \sum_{j=1}^{t-1}(\alpha d_j+\beta d_j^2+(\theta-\omega_i/\xi) c_j')-\frac{\ln 2}{\tau_d}(t-T_k)^+\right)-\mu_it-\omega_i C_{\max}\le \\ & -\xi\left( \sum_{j=1}^{t-1}(\alpha d_j+\beta d_j^2+(\theta-\omega_i/\xi) c_j^*)-\frac{\ln 2}{\tau_d}(t-T_k)^+\right)-\mu_it-\omega_i C_{\max}
\end{align*}
or equivalently
$$ 
\sum_{k=1}^{t-1}(\alpha d_k+\beta d_k^2+(\theta-\omega_i/\xi) c_k')\ge   \sum_{k=1}^{t-1}(\alpha d_k+\beta d_k^2+(\theta-\omega_i/\xi) c_k^*),\ \ i=1,\dots,n
$$
Since both schedules have an identical radiation dose vector, we need to show that
$$ 
(\theta-\omega_i/\xi)\sum_{k=1}^{t-1} c_k'\ge (\theta-\omega_i/\xi)\sum_{k=1}^{t-1} c_k^* \ \ t=1,\dots,N+1 \text{ and } \ \ i=1,\dots,n.
$$
We assume that  $\theta-\omega_i/\xi\ge0$ for all $i=1,\dots,n$; therefore, the above inequality is implied by 
$$\sum_{k=1}^{t-1} c_k'\ge \sum_{k=1}^{t-1} c_k^*\text{ for all } t=1,\dots,N+1.$$
Note that $c_i'=\min\{c_{\max},c_i^*+c_j^*\}$, $c_j'=c_j^*-(c_i'-c_i^*)$ and $c_k'=c_k^*$ for $j>i$ and $ k\notin i,j$; therefore, we have $c_i'>c_i^*$, which implies that $\sum_{t=1}^{l}c_t'>\sum_{t=1}^{l}c_t^*$ for $ l: i\le l< j$ and $\sum_{t=1}^{l}c_t'=\sum_{t=1}^{l}c_t^*$ for $ l\ge j$ and $l< i$.  

Next, note that as a result of Lemma \ref{lemmasum}, we know that $\sum_{t=1}^{N}c_t^*=C_{\max}$. Hence, we can repeat the above procedure to improve the optimal solution until we get $c_1^*=\dots=c_k^*=c_{\max}$ for $k=\lfloor C_{\max}/c_{\max}\rfloor$, where no more improvement can be achieved. At this step, we have two scenarios: first, $\lfloor C_{\max}/c_{\max}\rfloor\in \mathbb{N}$, where we must have $c^*_{k+1}=\dots=c^*_{N}=0$, which proves our result; second, $\lfloor C_{\max}/c_{\max}\rfloor\not = C_{\max}/c_{\max}$, where $\sum_{t=k+1}^{N}c_t=C_{\max}-kc_{\max}<c_{\max}$. In this case, by using the same contradiction as before, we can show that if we choose $c_{k+1}<C_{\max}-kc_{\max}$, then we can always construct a feasible solution with a smaller objective function.

We can use a similar approach when $\theta\xi<\min_{i=1,\dots,n} \omega_i$. In this case, we have $\theta-\omega_i/\xi<0,\ \  i=1,\dots,n$. Therefore, for any schedule with $c_i^*<c_{\max}$ and $c_j^*>0$ for $i>j$, we can construct a new feasible solution $\vec{c'}=\{c_1',\dots,c_{N}'\}$, where $c_i'=\min\{c_{\max},c_i^*+c_j^*\}>c_i^*$, $c_j'=c_j^*-(c_i'-c_i^*)<c_j^*$ and $c_k'=c_k^*$ for $ k\not=i,j$ with a smaller objective function. Similar to the previous case, here we require that\small
$$ \sum_{k=1}^{t-1}(\alpha d_k+\beta d_k^2+(\theta-\omega_i/\xi) c_k')\ge   \sum_{k=1}^{t-1}(\alpha d_k+\beta d_k^2+(\theta-\omega_i/\xi) c_k^*),\ \  i=1,\dots,n$$\normalsize
since $\theta-\omega_i/\xi<0\ \  i,$ which equivalently shows that
$$\sum_{k=1}^{t-1} c_k'\le \sum_{k=1}^{t-1} c_k^*\text{ for all } t=1,\dots,N+1.$$
This inequality is implied by $\sum_{t=1}^{l}c_t'<\sum_{t=1}^{l}c_t^*$ for $ l: j\le l< i$ and $\sum_{t=1}^{l}c_t'=\sum_{t=1}^{l}c_t^*$ for $ l\ge i$ and $ l< j$. The rest is similar to the situation where $\theta\xi>\max_{i=1,\dots,n} \omega_i$. 
\end{proof}
{\color{black}
\subsection{DP algorithm for the case of dynamic radio-sensitivity parameters} \label{algorithm-dynamic}
The structure of the 6D DP algorithm is similar to the 4D case. However, we need to store information related to two additional state dimensions, which are $\{\hat{U}_t,\hat{V}_t\}$, leading to a larger number of rows and columns in our DP table. Hence, the 6D DP algorithm has a higher computational and space complexity compared to the 4D case. To ease this additional burden, we derive and use the following structural results on the optimal solution to significantly reduce the height of the DP table by only storing information related to promising states and skipping those that lead to sub-optimal solutions. 
\begin{lemma}\label{lemma-structural}
Consider two arbitrary states at time stage $t$, $\mathcal{B}_{1,t}=\{u_1,v_1,s_1,w_1,\hat{u}_1,\hat{v}_1\}$ and $\mathcal{B}_{2,t}=\{u_2,v_2,s_2,w_2,\hat{u}_2,\hat{v}_2\}$ where $\hat{J}(\mathcal{B}_{2,t})\le\hat{J}(\mathcal{B}_{1,t})$. Then we can discard the row associated with state $\mathcal{B}_{1,t}$ in our DP table if $u_2\le u_1$, $v_2\le v_1$, $s_1=s_2$, $w_2\ge w_1$, $\hat{u}_2\ge \hat{u}_1$ and $\hat{v}_2\ge \hat{v}_1$.
\end{lemma}
\begin{proof}
We prove this result by contradiction. Suppose there exists an optimal path initiated from $\mathcal{B}_{1,t}$, then we construct a solution using $\mathcal{B}_{2,t}$ with a smaller objective value. To this end, let the optimal solution associated with $\mathcal{B}_{1,t}$ be $\vec{d}^*=\{d_1^*,\dots,d_t^*,\dots,d_N^*\}$ and $\vec{c}^*=\{c_1^*,\dots,c_t^*,\dots,c_N^*\}$ and the path resulting in $\hat{J}(\mathcal{B}_{2,t})$ be $\{d_1',\dots,d_t'\}$ and $\{c_1',\dots,c_t'\}$. Then we can construct solution $\{d_1',\dots,d_t',d_{t+1}^*,\dots,d_N^*\}$ and $\{c_1',\dots,c_t',c_{t+1}^*,\dots,c_N^*\}$  with a smaller or equal objective value. It suffices to establish that
$$
\hat{J}(\mathcal{B}_{2,k})\le\hat{J}(\mathcal{B}_{1,k}), \forall k=t+1,\dots,N+1
$$
We show the validity of the inequality above for $k=t+1$, and the proof for $k=t+2,\dots,N+1$ is similar. The inequality $\hat{J}(\mathcal{B}_{2,t+1})\le\hat{J}(\mathcal{B}_{1,t+1})$ is equivalent to
$$
 \hat{J}(\mathcal{B}_{2,t})+X_0^\xi \sum_{i=1}^{n}p_i e^{-\xi\times\hat{\Psi}_{2,i,t+1}+\mu_i(T-t-1)-\omega_i C_{\max}}
\le
\hat{J}(\mathcal{B}_{1,t})+X_0^\xi \sum_{i=1}^{n}p_i e^{-\xi\times\hat{\Psi}_{1,i,t+1}+\mu_i(T-t-1)-\omega_i C_{\max}}.
$$\normalsize
Note that the above inequality is implied by 
$$
\hat{J}(\mathcal{B}_{2,t})\le \hat{J}(\mathcal{B}_{1,t}) \text{ and } \hat{\Psi}_{2,i,t+1} \ge \hat{\Psi}_{1,i,t+1} \ \ \forall i.
$$
The inequality $\hat{\Psi}_{2,i,t+1} \ge \hat{\Psi}_{1,i,t+1}$ follows from 
\begin{align}
\label{eq:Psi_Ineq}
&\hat{u}_{2}+\alpha_{2,t+1}d_{t+1}^*+ \hat{v}_{2}+\beta_{2,t+1}(d_{t+1}^*)^2+(\theta-\omega_i/\xi) (s_2+c_{t+1}^*)+\psi(w_{2}+d_{t+1}^*c_{t+1}^*)\ge\\
&\quad \hat{u}_{1}+\alpha_{1,t+1}d_{t+1}^*+ \hat{v}_{1}+\beta_{1,t+1}(d_{t+1}^*)^2+(\theta-\omega_i/\xi) (s_1+c_{t+1}^*)+\psi(w_{1}+d_{t+1}^*c_{t+1}^*).\nonumber
\end{align}
Note that we have $s_1=s_2$, $w_2\ge w_1$, $\hat{u}_2\ge \hat{u}_1$, and $\hat{v}_2\ge \hat{v}_1$, and we can conclude that $y_{2,t}\ge y_{1,t}$. Therefore, to prove inequality \eqref{eq:Psi_Ineq}, it suffices to show
$$
\alpha_{2,t+1}d_{t+1}^*+\beta_{2,t+1}(d_{t+1}^*)^2\ge \alpha_{1,t+1}d_{t+1}^*+\beta_{1,t+1}(d_{t+1}^*)^2.
$$
Since we always have $OER_\alpha\ge1$ and $OER_\beta\ge1$, and the inequality $y_{2,t}\ge y_{1,t}$ holds, we can use \eqref{eq-OER} to complete the proof
$$
\alpha_{2,t+1}\ge \alpha_{1,t+1} \text{ and } \beta_{2,t+1}\ge \beta_{1,t+1}.
$$

\end{proof}

We next use the result of Lemma \ref{lemmasum} to discard another set of sub-optimal states during the state enumerations at each stage.

{\bf{Observation 1}:} Consider an arbitrary state at time stage $t$, $\mathcal{B}_{t}=\{U_t,V_t,S_t,W_t,\hat{U}_t,\hat{V}_t\}$. Then we can discard $\mathcal{B}_{t}$ if $(N-t)c_{\max}<C_{\max}-S_t$.

Using the above results, we use Algorithm \ref{alg2} to solve problem \eqref{eq-obj1}--\eqref{eq-checon1} with dynamic radio-sensitivity parameters $\alpha$ and $\beta$. 

\begin{algorithm}
\caption{DP algorithm to solve problem \eqref{eq-obj1}--\eqref{eq-checon1} with dynamic $\alpha$ and $\beta$}
\begin{algorithmic}[1]\label{alg2}
  \scriptsize
  \STATE {\color{black}Create a data table with headers: $T$=\{$day, U,V,S,W,\hat{U},\hat{V},d,c,obj,id,track$\}}
  \STATE {\color{black}$y_1\leftarrow y_{\max}-\iota\times\sqrt[3]{\frac{3X_0}{4\pi\rho}}, counter\leftarrow 1$ and compute $\alpha_1$ and $\beta_1$ based on \eqref{eq-OER}}
  \STATE {\color{black}For $i\in \{0:d_{\min}:d_{\max}\}$ do:}
  \STATE {\color{black}$\mbox{ }$ For $j\in \{0:c_{\min}:c_{\max}\}$ do:}
  \STATE {\color{black}$\mbox{ }\quad$ Create a temporary data table with headers: $T_{\text{te}}=\{day, U,V,S,W,\hat{U},\hat{V},d,c,obj,id,track\}$}
  \STATE {\color{black}$\mbox{ }\quad$ $T_{\text{te}}.day\leftarrow 1,T_{\text{t}}.U\leftarrow i, T_{\text{te}}.V\leftarrow i^2, T.S\leftarrow j, T_{\text{te}}.W\leftarrow i\times j, T_{\text{te}}.\hat{U}\leftarrow \alpha_1\times i, T_{\text{te}}.\hat{V}\leftarrow \beta_1\times i^2,T_{\text{te}}.d\leftarrow i, T_{\text{te}}.c\leftarrow j$, $T_{\text{te}}.id\leftarrow counter$, $T_{\text{te}}.track\leftarrow 0$, and
  $$T_{\text{te}}.obj\leftarrow X_0^\xi \sum_{i=1}^{n}p_ie^{\mu_iT-\omega_iC_{\max}}\left(1+e^{-\mu_i}+ e^{-\xi\left( \alpha_1 i+\beta_1 i^2+(\theta-\omega_i/\xi) j+\psi ij-\frac{\ln 2}{\tau_d}(2-T_k)^+\right)-2\mu_i}\right)  $$}
  \STATE {\color{black}$\mbox{ }\quad$ $counter\leftarrow counter+1$}
  \STATE {\color{black}$\mbox{ }\quad$ $concatenate(T,T_{\text{te}})$}
  \STATE {\color{black}For $t\in \{2:1:N\}$ do:}
  \STATE {\color{black}$\mbox{ }$ Set $T'=\{\text{all entities in table } T|T.day=t-1\}$}
  \STATE {\color{black}$\mbox{ }$ For $k\in\{1:1:\text{height}(T')\}$ do:}
  \STATE {\color{black}$\mbox{ }\quad$ $y_t\leftarrow y_{\max}-\iota\times \sqrt[3]{\frac{3X_0}{4\pi\rho}}\exp\left(-\frac{1}{3}\left( T'.\hat{U}(k)+T'.\hat{V}(k)+\theta T'.S(k)+\psi T'.W(k)+\frac{\ln 2}{\tau_d}(t-T_k)^+\right) \right)
  $  and compute $\alpha_t$ and $\beta_t$ using equations \eqref{eq-OER}}
  \STATE {\color{black}$\mbox{ }\quad$ For $i\in \{0:d_{\min}:d_{\max}\}$ do:}
  \STATE {\color{black}$\mbox{ }\quad\quad$ For $j\in \{0:c_{\min}:c_{\max}\}$ do:}
  \STATE {\color{black}$\mbox{ }\quad\quad\quad$ Create a temporary data table with headers: $T_{\text{te}}=\{day, U,V,S,W,\hat{U},\hat{V},d,c,obj,id,track\}$}
  \STATE {\color{black}$\mbox{ }\quad\quad\quad$ Set 
  $$
  T_{\text{te}}.day\leftarrow t, T_{\text{te}}.U\leftarrow T'.U(k)+i, T_{\text{te}}.V\leftarrow T'.V(k)+i^2, T_{\text{te}}.S\leftarrow T'.S(k)+j, T_{\text{te}}.W\leftarrow T'.W(k)+i\times j, 
  $$ 
  $$
  T_{\text{te}}.\hat{U}\leftarrow T'.\hat{U}(k)+i\times\alpha_t, T_{\text{te}}.\hat{V}\leftarrow T'.\hat{V}(k)+i^2\times\beta_t, T_{\text{t}}.d\leftarrow i, T_{\text{te}}.c\leftarrow j,
  $$ 
  $$
  \Psi_{i,t}=\mu_i\cdot(T-t)-\omega_iC_{\max}-\xi\times( T_{\text{te}}.\hat{U}+T_{\text{te}}.\hat{V}+(\theta-\omega_i/\xi) T_{\text{te}}.S+\psi T_{\text{te}}.W-\frac{\ln 2}{\tau_d}(t-T_k)^+)
  $$
  $$ 
  T_{\text{te}}.obj\leftarrow T'.obj(k)+\begin{cases} X_0^\xi \sum_{i=1}^{n}p_ie^{\Psi_{i,t}} & t\le N-1 \\
  X_0^\xi \sum_{i=1}^{n}p_ie^{\Psi_{i,t}}+\hat{g}(T_{\text{te}}.\hat{U},T_{\text{te}}.\hat{V},T_{\text{te}}.S,T_{\text{te}}.W) & t=N \end{cases}
  $$ }
  \STATE {\color{black}$\mbox{ }\quad\quad\quad$ If (all feasibility constraints satisfied and $(N-t)c_{\max}\ge C_{\max}-T_{\text{te}.S}$):}
  \STATE {\color{black}$\mbox{ }\quad\quad\quad\quad$ Set 
  $$
  \vec{index}=\{\forall \ \ T.id|T.day=t,T_{\text{te}}.U\le T.U,T_{\text{te}}.V\le T.V,T_{\text{te}}.S=T.S,T_{\text{te}}.W\ge T.W,T_{\text{te}}.\hat{U}\ge T.\hat{U} ,T_{\text{te}}.\hat{V}\ge T.\hat{V}\}
  $$}
  \STATE {\color{black}$\mbox{ }\quad\quad\quad\quad$ If ($\vec{index}==NULL$):} 
  \STATE {\color{black}$\mbox{ }\quad\quad\quad\quad\quad$ $T_{\text{te}}.id\leftarrow counter $}
  \STATE {\color{black}$\mbox{ }\quad\quad\quad\quad\quad$ $T_{\text{te}}.track\leftarrow T'.id(k) $}
  \STATE {\color{black}$\mbox{ }\quad\quad\quad\quad\quad$ $concatenate(T,T_{\text{te}})$}
  \STATE {\color{black}$\mbox{ }\quad\quad\quad\quad\quad$ $counter\leftarrow counter+1$}
  \STATE {\color{black}$\mbox{ }\quad\quad\quad\quad$ Else:}
  \STATE {\color{black}$\mbox{ }\quad\quad\quad\quad\quad$ For $k\in\vec{index}$}
  \STATE {\color{black}$\mbox{ }\quad\quad\quad\quad\quad\quad$ If ($T_{\text{te}}.obj\le T.obj(T.id==index(k))$): }
  \STATE {\color{black}$\mbox{ }\quad\quad\quad\quad\quad\quad\quad$ Remove row with ID $index(k)$ from table T}
  \STATE {\color{black}$\mbox{ }\quad\quad\quad\quad\quad\quad\quad$ $check\leftarrow 1$}
  \STATE {\color{black}$\mbox{ }\quad\quad\quad\quad\quad$ If check=1}
  \STATE {\color{black}$\mbox{ }\quad\quad\quad\quad\quad\quad$ $T_{\text{te}}.id\leftarrow counter $}

  \STATE {\color{black}$\mbox{ }\quad\quad\quad\quad\quad\quad$ $T_{\text{te}}.track\leftarrow T'.id(k) $}
  \STATE {\color{black}$\mbox{ }\quad\quad\quad\quad\quad\quad$ $concatenate(T,T_{\text{te}})$}
  \STATE {\color{black}$\mbox{ }\quad\quad\quad\quad\quad\quad$ $counter\leftarrow counter+1$}
  \STATE {\color{black}$id^*\leftarrow\left\lbrace  T.id|T.obj=\min\{T.obj|T.day=N\} \right\rbrace $}
  \STATE {\color{black}Return $d^*_N\leftarrow T.d(T.id==id^*)$, $c^*_N\leftarrow T.c(T.id==id^*)$}
  \STATE {\color{black}For $t\in \{N-1:-1:1\}$ do:}
  \STATE {\color{black}$\mbox{ }$ $id^*\leftarrow T.track(T.id==id^*)$}
  \STATE {\color{black}$\mbox{ }$ Return $d^*_{t}\leftarrow T.d(T.id==id^*)$, $c^*_t\leftarrow T.c(T.id==id^*)$}
\end{algorithmic}
\end{algorithm}
}
 
\end{document}